\documentclass[reqno]{amsart}
\usepackage{float}
\usepackage{dsfont}
\usepackage{amsmath}
\usepackage{graphicx} 
\usepackage{latexsym}
\usepackage{amsfonts}
\usepackage{amssymb}
\usepackage{color}
\usepackage{xcolor}
\theoremstyle{plain}
\newtheorem{theorem}{Theorem}
\newtheorem{corollary}[theorem]{Corollary}
\newtheorem{lemma}[theorem]{Lemma}
\newtheorem{proposition}[theorem]{Proposition}
\theoremstyle{definition}

\newtheorem{definition}[theorem]{Definition}

\newtheorem{remark}[theorem]{Remark}
\newtheorem*{remark*}{Remark}

\newcommand{\dH}{\frac{d H}{d s}}
\newcommand{\dG}{\frac{d G}{d s}}
\newcommand{\cH}{\mathcal H}
\newcommand{\cR}{\mathcal R}
\newcommand{\N}{\mathbb N}
\newcommand{\R}{\mathbb R}
\newcommand{\pr}{\mathbf P}
\newcommand{\e}{\mathbf E}
\newcommand{\E}{\mathbf E}

\newcommand{\rf}{\rfloor}
\newcommand{\lf}{\lfloor}

\renewcommand{\Im}{\mathop{\mathrm{Im}}\nolimits}
\renewcommand{\Re}{\mathop{\mathrm{Re}}\nolimits}

\begin{document}
\title[Expansions for random walks conditioned to stay positive]
    {Expansions for random walks conditioned to stay positive}

\author[Denisov]{Denis Denisov}
\address{Department of Mathematics, University of Manchester, UK}
\email{denis.denisov@manchester.ac.uk}

\author[Tarasov]{Alexander Tarasov}
\address{Faculty of Mathematics, Bielefeld University, Germany}
\email{atarasov@math.uni-bielefeld.de}

\author[Wachtel]{Vitali Wachtel}
\address{Faculty of Mathematics, Bielefeld University, Germany}
\email{wachtel@math.uni-bielefeld.de}
\begin{abstract}
We consider a one-dimensional random walk $S_n$ with i.i.d. increments with zero mean and finite variance. We study the asymptotic expansion for the tail distribution $\mathbf P(\tau_x>n)$ of the first passage times 
$\tau_x:=\inf\{n\ge1:x+S_n\le0\}$ for $\ x\ge0.$
We also derive asymptotic expansion for local probabilities $\mathbf P(S_n=x,\tau_0>n)$. 
Studying the asymptotic expansions we obtain a sequence of discrete polyharmonic functions and obtain analogues of renewal theorem for them. 
\end{abstract}
    
    % \version\vspace{1cm}
    
    \maketitle
    {\scriptsize
    %TCIMACRO{\TeXButton{toc}{\tableofcontents}}%
    %BeginExpansion
    %\tableofcontents
    %EndExpansion
    }
%%%%%%%%%%%%%%%%%%%%%%%%%%%%%%%%%%%%%%%%%%%%%%%%%%%%%%%%%%%%%%%%%%%%%%%%%%%%%%%%%%%%%%%%%%%%%%%%%%%%%%%%%%%%%%%%%%%%%%%%%%%%%%%%%%%%%%%%%%%%%%%%%%%%%%
\section{Introduction.}
Let  $\{X_k\}$ be a sequence of  independent random variables with zero mean 
$\e X_1=0$ and finite variance $\e X_1^2\in(0,\infty)$. 
Consider a random walk 
 $\{S_n; n\ge0\}$ defined as follows, 
 $S_0=0$ and 
$$
S_n:=X_1+X_2+\ldots+X_n,\ n\ge1.
$$
One of the classical tasks in the theory of random walks is the study of first passage times 
$$
\tau_x:=\inf\{n\ge1:x+S_n\le0\},\ x\ge0.
$$
The exact distribution of $\tau_x$ is known only 
in some special cases that we will shortly discuss below. 
Hence there is much interest for approximations for the tail distribution $\pr(\tau_x>n)$. 
One of the important results in this direction 
says that 
\begin{align}\label{eq:taux_tail}
\pr(\tau_x>n)\sim \frac{V_1(x)}{\sqrt{\pi n}}, \quad \text{uniformly in } x=o(\sqrt{n}), n\to \infty,
\end{align}
for some (discrete) harmonic function $V_1$. 
Besides its own interest and various applications these asymptotics play an important 
role in the studies of random walks conditioned to stay positive~\cite{BD94}.

One of the main purposes of the paper is to improve \eqref{eq:taux_tail} by  obtaining an asymptotic expansion for $\pr(\tau_x>n)$. 
In doing this we derive  a sequence of polyharmonic functions $V_j(x)$ and study its properties.  This requires development of a new more accurate approach that considerably improves over Tauberian theorems in this situation. 
In the multidimensional situation there are some 
results analogous to~\eqref{eq:taux_tail}, see~\cite{DW15} and follow-up papers. 
We believe that these studies will help 
to obtain analogous asymptotic expansions in the 
multidimensional situation. 

We now describe some known results that will bring us to a starting point of our reasoning. 
The distribution of $\tau_0$ can be studied by using the Spitzer formula that follows from the Wiener-Hopf factorisation. The Spitzer formula provides the following exact
expression for the generating function of $\tau_0$ (see, for example, \cite[Theorem~8.9.1,~p.~376]%
{BGT})
\begin{equation}
\label{eq:WF1} 
1-\mathbf{E}s^{\tau_0}
=\exp\left\{-\sum_{n=1}^{\infty }\frac{s^{n}}{n}\mathbf{P}(S_{n}\leq 0)\right\}.
\end{equation}
Using this expression one can find 
an exact expression for $\pr(\tau_0>n)$ when the distribution of $X_1$ is symmetric and has no atoms. In this case   $\mathbf{P}(S_{n}\leq 0)=\frac{1}{2}$ for all $n$ and, consequently,
$$
1-\mathbf{E}s^{\tau_0}=(1-s)^{1/2}.
$$
Equivalently,
$$
\sum_{n=0}^\infty\pr(\tau_0>n)s^n
=\frac{1-\mathbf{E}s^{\tau_0}}{1-s}=(1-s)^{-1/2}.
$$
Therefore,
\begin{equation}
\label{eq:exact}
\pr(\tau_0>n)=(-1)^n{-\frac{1}{2}\choose n}=:a^{(1)}_n,\quad n\ge0.
\end{equation}
Here recall that fractional binomial coefficients are defined by the formula
$$
{\gamma\choose n}=\frac{\gamma(\gamma-1)\ldots(\gamma-n+1)}{n!}.
$$
When the distribution of $X$ is not symmetric one  can use the convergence $\pr(S_n\le 0)\to \frac{1}{2}$ instead of exact equality.  
Then the asymptotics for the tail of $\tau$, see e.g.~\cite[Chapter XII, Theorem 1a]{Feller71}, is given by 
\begin{equation}
\label{eq:tau0_tail}
\pr(\tau_0>n)\sim\frac{K}{\sqrt n}\quad\text{as }n\to\infty.
\end{equation}
%One of the main purposes of the paper is to improve \eqref{eq:tau0_tail} by  obtaining an asymptotic expansion for $\pr(\tau_0>n)$. 
%We will also obtain asymptotic expansions in the case of arbitrary starting point $x$. 

In order to explain approach that we use, we rewrite \eqref{eq:WF1} in the following way:
\begin{equation}
\label{eq:WF2}
\sum_{n=0}^\infty\pr(\tau_0>n)s^n
=\frac{1-\mathbf{E}s^{\tau_0}}{1-s}
=(1-s)^{-1/2}\exp\left\{\sum_{n=1}^\infty\frac{s^n}{n}\Delta_n\right\},
\end{equation}
where
$$
\Delta_n:=\frac{1}{2}-\pr(S_n\le 0).
$$
Since  $\e X_1=0$ and $\e X_1^2\in(0,\infty)$ the convergence 
$\pr(S_n\le0)\to1/2$ holds and, by the Spitzer-R\'{o}sen theorem (see \cite[%
Theorem~8.9.23,~p.~382]{BGT}),
$$
\sum_{n=1}^\infty\frac{\Delta_n}{n}<\infty.
$$
Then, applying a Tauberian theorem to the generating function in \eqref{eq:WF2}, we infer that 
\begin{equation}
\label{eq:tau0_tail_1}
\pr(\tau_0>n)\sim 
\exp\left\{\sum_{n=1}^\infty\frac{\Delta_n}{n}\right\}a^{(1)}_n
\quad n\to\infty.
\end{equation}
These  computations together with \eqref{eq:WF2} suggests that if one wants to obtain 
terms of higher order then one needs to learn more about the  behaviour of $\Delta_n$. Remarkably this information is available from the asymptotic expansions in the central limit theorem. 

To formulate our results, we introduce for every $j\ge1$ the sequence
$$
a_n^{(j)}:= [s^n](1-s)^{j-3/2}=(-1)^n{j-\frac{3}{2}\choose n},\ n\ge0,
$$
where $[s^n]f(s)=f_n$ for $f(s) = \sum\limits_{n=0}^\infty f_n s^n$.

% These coefficients appear in the binomial theorem, that is  
% \begin{equation*}
% \sum_{n=0}^\infty a_n^{(j)}s^n=(1-s)^{j-3/2},\quad j\ge1. 
% \end{equation*}
It will be shown later that $a_n^{(j)} \sim \beta_j n^{1/2-j}$
for some $\beta_j\neq0$.
This implies that asymptotic expansions in terms of $n^{-j-1/2}$ can be rewritten via $a_n^{(j)}$, and vice versa. There are several arguments why we shall use 
sequences $\{a_n^{(j)}\}$. First of all, $a_n^{(1)}$ naturally appears in $\pr (\tau_0 > n)$ for symmetric case and also as the leading term in the general finite-variance case, see \eqref{eq:tau0_tail_1}. Secondly, they satisfy relation
\begin{align*}
    a_n^{(j)} - a_{n-1}^{(j)} = a_n^{(j+1)}.
\end{align*}
This relation allows one to compute exactly  the tail sums 
\begin{align*}
    \sum_{k=n}^\infty a_k^{(j+1)} = 
    \sum_{k=n}^\infty (a_k^{(j)} - a_{k-1}^{(j)}) = - a_{n-1}^{(j)}, 
\end{align*}
while the sum
\begin{align*}
    \sum_{k=n}^\infty \frac{1}{n^{\alpha}}
\end{align*}
can be calculated only asymptotically. Thus, the sequences $\{a_n^{(j)}\}$ are more convenient also from the technical point of view.
\begin{theorem}
\label{thm:tau_0}
Assume that $\e|X_1|^{r+3}$ is finite for some integer $r\ge0$. Assume also that either the distribution of $X_1$ is lattice with maximal span $1$ or 
it holds $\lim\sup_{|t|\to\infty}\left|\e~e^{itX_1}\right|<1$. Then there exist numbers $\nu_0,\nu_1,\ldots,\nu_{\lfloor r/2\rfloor}$ such that
$$
\pr(\tau_0>n)=\sum_{j=1}^{\lfloor r/2\rfloor+1}
\nu_ja^{(j)}_n
+
O\left(\frac{\log^{{\delta_r}}  n}{n^{(r+2)/2}}\right), \quad n\to\infty, 
$$
where $\delta_r = \mathds{1}_{\{r \text{ is odd} \}}$.

Furthermore, there exist numbers
$\overline{\nu}_0,\overline{\nu}_1,\ldots,\overline{\nu}_{\lfloor r/2\rfloor}$ such that 
$$
\pr(\overline{\tau}_0>n)=\sum_{j=1}^{\lfloor r/2 \rfloor+1} 
\overline{\nu}_ja^{(j)}_n
+
O\left(\frac{\log^{{\delta_r}}  n}{n^{(r+2)/2}}\right),
$$
where 
$$
\overline{\tau}_0:=\inf\{n\ge1:S_n<0\}.
$$
\end{theorem}

It turns out that in order to derive expansions for $\pr(\tau_x>n)$ one has to use asymptotic expansions in the local central limit theorem.  Possibly, the best illustration for this is given by left-continuous random walks. Recall that the random walk $S_n$ is called left-continuous if $\pr(X_1\in\{-1,0,1,2,\ldots\})=1$.  It is well-known (see e.g. \cite[Theorem 12.8.4]{Borovkov13}) that in this case one has the relation 
\begin{equation}
\label{eq:lcont}
\pr(\tau_x=n)=\frac{x}{n}\pr(S_n=-x),\quad x,n\ge1.
\end{equation}
This equality allows one to prove the following result.
\begin{theorem}
\label{thm:left-cont}
Assume that $S_n$ is left-continuous and that $\e|X_1|^{r+3}$ is finite for some integer $r \ge0$. Then there exist polynomials $V_1(x),V_2(x),\ldots,V_{\lfloor r/2\rfloor+1}(x)$ (every $V_k$ is of degree $2k-1$) such that 
$$
\pr(\tau_x>n)=\sum_{j=1}^{\lfloor r/2\rfloor+1} 
V_j(x)a^{(j)}_n+O\left(\frac{1+x^{r+2}}{n^{(r+2)/2}}\right), \quad n\to\infty, 
$$
uniformly in $x\ge0$.
\end{theorem}
Our approach to expansions for $\pr(\tau_x>n)$ in general case is based on expansions for conditional probabilities $\pr(S_n=x,\overline{\tau}_0>n)$. We believe that these expansions are also of independent interest.

\begin{theorem}
\label{thm:cond_prob}
Assume that $\e|X_1|^{r+3}$ is finite for some integer $r\ge1$.
If $S_n$ is lattice with maximal span $1$ then there exist functions
$U_1$, $U_2$,\ldots, $U_{\lf \frac{r+1}{2}\rf}$ and  $\overline{U}_1, \overline{U}_2,\ldots, \overline{U}_{\lf \frac{r+1}{2}\rf}$ such that, uniformly in $x$,
\begin{equation}
\label{eq:cp1}
\pr(S_n=x,\tau_0>n)=\sum_{j=1}^{\lf \frac{r+1}{2}\rf} U_j(x)a^{(j+1)}_n
+
O\left(\frac{(1+x)^{r+1}}{n^{(r+3)/2}}\log^{{\lf \frac{r}{2}\rf}}n\right)
\end{equation}
and 
\begin{equation}
\label{eq:cp2}
\pr(S_n=x,\overline{\tau}_0>n)=\sum_{j=1}^{\lf \frac{r+1}{2}\rf}  \overline{U}_j(x)a^{(j+1)}_n
+
O\left(\frac{(1+x)^{r+1}}{n^{(r+3)/2}}\log^{\lf \frac{r}{2}\rf }n\right).
\end{equation}
For the functions $U_j$ and $\overline{U}_j$ one has
\begin{align*}
    |U_j(x)|+|\overline{U}_j(x)| \le C(1+|x|)^{2j-1},\quad j\ge1.
\end{align*}

If the distribution of $X_1$ is absolutely continuous then the expansions \eqref{eq:cp1} and \eqref{eq:cp2} hold for densities of measures $\pr(S_n\in\cdot,\tau_0>n)$ and  
$\pr(S_n\in\cdot,\overline{\tau}_0>n)$ respectively.
\end{theorem}

These expansions are refinements of the results in \cite{VW09},
where the local limit theorems for conditioned random walks were proven. 

Now we can state our main result on asymptotic expansions for the tails of stopping times $\tau_x$.
\begin{theorem}
\label{thm:tau_x}
Assume that $\e|X_1|^{r+3}$ is finite for some integer $r\ge2$.
Assume also that $S_n$ is either lattice or absolutely continuous. Then there exist functions
$V_1$, $V_2$,\ldots, $V_{\lf r/2 \rf+1}$ such that, uniformly in $x\ge0$,
$$
\pr(\tau_x>n)=\sum_{j=1}^{\lf r/2 \rf+1}V_j(x)a_n^{(j)}+O\left(\frac{(1+x)^{r-2}}{n^{(r-1)/2}}\log^{\lf \frac{r+1}{2}\rf}n\right).
$$
\end{theorem}
The main asymptotic term for $\pr(\tau_x>n)$ is known in the literature, see Corollary 3 in Doney~\cite{Doney12}, where the main term is found for all asymptotically stable random walks.

Grama and Xiao~\cite{GramaXiao21} have studied the rate of convergence for conditioned random walks, which can be seen as an analogue of the classical Berry-Esseen inequality.  

Since 
$$
\{\tau_x>n\}=\left\{\min_{k\le n}S_k>-x\right\}
=\left\{\max_{k\le n}(-S_k)<x\right\},
$$
The question on asymptotic expansions for the tail of $\tau_x$ is equivalent to the question of asymptotic expansions for the running maximum of a random walk. This  problem was studied rather intensively in 60's of the last century. Koroljuk~\cite{Koroljuk62}   derived some expansions for the probability 
$\pr\left(\max_{k\le n}S_k<x\sqrt{n}\right)$ for every fixed $x>0$. Furthermore, 
Theorem~6 in Borovkov~\cite{Borovkov62} contains a full expansion for
$\pr\left(\max_{k\le n}S_k<x\right)$ for $x=o(n)$ under the assumption that $X_1$ has finite exponential moment. The remainder term in that result is of order $O(e^{-\gamma x})$ with some $\gamma>0$. Thus, Borovkov's expansion is not applicable in the case of bounded $x$. Nagaev~\cite{Nagaev70}  derived an expansion for $\pr\left(\max_{k\le n}S_k<x\sqrt{n}\right)$ under the assumption that the moment of order $m$ is finite. The remainder term in that paper is given by
$$
O\left(\min\left\{\frac{1}{\sqrt{n}}, (1+(x/\sqrt{n})^{1-m})n^{1-m/2}\log^2n\right\}\right).
$$
This implies that his expansion is also not applicable in the case of fixed $x$.
It is also worth mentioning that the coefficients in expansions from \cite{Borovkov62, Koroljuk62, Nagaev70} are polynomials of $x$. This fact indicates already that these results can not be valid for fixed values of $x$. (It remains to notice that the renewal function $V$ is not a linear function in general.)

We now turn to the properties of the functions $U_k$, $\overline{U}_k$ and $V_k$, which show up as coefficients in our asymptotic expansions.  
\begin{theorem}
\label{thm:polyharm}
Assume that the conditions of Theorems \ref{thm:cond_prob} and \ref{thm:tau_x}
are valid. Then the function $V_k$ is polyharmonic of order $k$ for $\{S_n\}$ killed at leaving $(0,\infty)$, that is,
$$
(P-I)^kV_k=0,
$$
where $I$ is the identity and the operator $P$ is defined by
$$
Pf(x)=\E[f(x+X);x+X>0].
$$
Moreover, every $U_k$ is polyharmonic of order $k$ for $\{-S_n\}$ killed at leaving $(0,\infty)$.
\end{theorem}
The polyharmonicity of the coefficients is not really surprising.
In the case of the Brownian motion, applying the reflection principle, one gets
\begin{align*}
\pr\left(\inf_{s\le t}B(s)>-x\right)
&=\pr\left(\inf_{s\le 1}B(s)>-\frac{x}{\sqrt{t}}\right)\\
&=2\Phi\left(\frac{x}{\sqrt{t}}\right)
=\sqrt{\frac{2}{\pi}}\int_0^{\frac{x}{\sqrt{t}}}e^{-u^2/2}du.
\end{align*}
Plugging here the series representation of $e^{-u^2/2}$, we obtain
$$
\pr\left(\inf_{s\le t}B(s)>-x\right)
=\sum_{k=0}^\infty\frac{(-1)^k}{k!(2k+1)2^k}
\frac{x^{2k+1}}{t^{k+1/2}}.
$$
Every function $x^{2k+1}$ is $k$-polyharmonic in the classical sense, i.e. $\frac{d^2k}{dx^{2k}}x^{2k+1}=0$. 
We are aware of only a few results on asymptotic expansions and polyharmonic functions for discrete time random walks. 
Chapon, Fusy and Raschel~\cite{CFR20} have constructed asymptotic expansions for some particular random walks in the positive quarter-plane. Nessmann~\cite{Nessmann22} has constructed all polyharmonic functions for two-dimensional random walks confined to the positive quarter. Nessmann~\cite{Nessmann23} has obtained asymptotic expansions for the so-called orbit-summable random walks in the positive quadrant. Finally, Elvey-Price, Nessmann and Raschel~\cite{ENR23} have derived asymptotic expansions for perturbed simple random walks in the positive quadrant. The intriguing feature of their expansions is the appearance of the logarithmic terms. All these results deal with random walks with small jumps, that is, all possible jumps belong to the set 
$\{-1,0,1\}^2$. 

Our last result deals with the asymptotic behaviour of polyharmonic functions introduced in our expansions. 
\begin{theorem}
\label{thm:V-asymp}
Assume that $\E e^{\lambda|X|}$ is finite for some $\lambda>0$ and that the distribution of $X$ is lattice. Then, for every $k$ there exist polynomials $P_k$, $Q_k$ and $\overline{Q}_k$ of degree $2k-1$
such that 
\begin{align*}
V_k(x)&=P_k(x)+O(e^{-\varepsilon x}),\\
U_k(x)&=Q_k(x)+O(e^{-\varepsilon x}),\\
\overline{U}_k(x)&=\overline{Q}_k(x)+O(e^{-\varepsilon x}).
\end{align*}
with some $\varepsilon>0$. 
\end{theorem}
These estimates can be seen as a generalisation of Gelfond's 
\cite{Gelfond64} estimate for the renewal function. 
%%%%%%%%%%%%%%%%%%%%%%%%%%%%%%%%%%%%%%%%%%%%%%%%%%%%%%%%%%%%%%%%%%%%%%%%%%%%%%%%%%%%%%%%%%%%%%%%%%%%%%%%%%%%%%%%%%%%%%%%%%%%%%%%%%%%%%%%%%%%%%%%%%%%%%
\vspace{6pt}

\textbf{Organisation of the paper.} In the subsequent short section we collect some known properties of sequences $\{a_n^{(j)}\}$. In Section~\ref{sec:clt_expan} we derive asymptotic expansions in the central limit theorem, which will play a crucial role in the proof of our main results; their proofs are postponed to Section~\ref{prop8.and.9}. In Section~\ref{sec:algebra} we introduce a special class of generating functions and develop technical tools of analysis of that class. This section is the most important part in our approach to the asymptotic expansions for conditioned walks. 
Section~\ref{sec:loc_prob} is devoted to the proof of expansions for local probabilities of conditioned walks.
In Sections~\ref{sec:tau0} and \ref{sec:tau_x} we prove our results on stopping times $\tau_x$, Theorems \ref{thm:tau_0},\ref{thm:left-cont} and \ref{thm:tau_x}. Finally, in Section~\ref{sec:poly} we prove polyharmonicity and derive asymptotic approximation for the coefficients in our expansions, see Theorems \ref{thm:polyharm} and \ref{thm:V-asymp}. 
%%%%%%%%%%%%%%%%%%%%%%%%%%%%%%%%%%%%%%%%%%%%%%%%%%%%%%%%%%%%%%%%%%%%%%%%%%%%%%%%%%%%%%%%%%%%%%%%%%%%%%%%%%%%%%%%%%%%%%%%%%%%%%%%%%%%%%%%%%%%%%%%%%%%%%

\section{Some properties of binomial coefficients}
For integer $j\ge 1$ recall the sequence 
\begin{align}\label{eq:anbin}
    a_n^{(j)}:= (-1)^n \binom{j-\frac{3}{2}}{n}, \quad n\ge 0,
\end{align}
arising from the binomial theorem 
\begin{align}\label{eq:anser}
    \sum_{n=0}^{\infty} a_n^{(j)}s^n = (1-s)^{j-3/2}.
\end{align}
% Next, let us rewrite $a_n^{(1)}$ in more suitable form:

% \begin{align}\label{eq:a_n1}
%     a_n^{(1)} = 
%     %(-1)^n \binom{-1/2}{n}
%  (-1)^{-n}\frac{-\frac{1}{2}(-\frac{1}{2} - 1)(-\frac{1}{2} - 2)\dots (-\frac{1}{2} - n + 1) }{n!}
% %&=\frac{\frac{1}{2}\cdot \frac{3}{2} \cdot \frac{5}{2}\dots \frac{2n-1}{2}}{n!} = \frac{(2n-1)!!}{2^n n!} =  \frac{(2n)!}{4^n (n!)^2} 
% = \frac{1}{4^n} \binom{2n}{n}. 
% \end{align}
% Asymptotics of the central binomial coefficient is known precisely \textcolor{red}{[ssilka?  Luke, Yudell L. (1969). The Special Functions and their Approximations, Vol. 1. New York, NY, USA: Academic Press, Inc. p. 35.? stireno s vikipedii]}
% \begin{align*}
%     \binom{2n}{n} = \frac{4^n}{\sqrt{\pi n}}\left(1-\frac{1}{8n}+\frac{1}{128 n^2}+\frac{5}{1024 n^3} + O(n^{-4})\right).
% \end{align*}
% In particular 
% \begin{align} \label{eq:a_n1asymp}
%     a_n^{(1)} = \frac{1}{\sqrt{\pi n}} + O(n^{-3/2})
% \end{align}
% Moreover one can write that for any fixed $m \in \mathbb{N}$ there are numbers $\beta_k$ such that
% \begin{align} \label{eq:a_n1dec}
%     a_n^{(1)} = \sum_{k=1}^m \frac{\beta_k}{n^{k-1/2}} + O\left(\frac{1}{n^{m+1/2}}\right)
% \end{align}
For these sequences we will make use of the following asymptotic expansion (see \cite[%
Lemma~17]{DenisovFitzgerald23}). For any $j \ge 1$  there is a sequence $\{\beta_k^{(j)}\}_{k=j}^\infty$ such that for any $m \ge j$ it holds
\begin{align} \label{eq:anasympdecomp}
    a_n^{(j)} = \sum_{k=j}^{m} \frac{\beta_k^{(j)}}{n^{k-1/2}} + O\left( 
    \frac{1}{n^{m+1/2}}
    \right),\quad n\to \infty, 
\end{align}
and furthermore 
\begin{align*}
    \beta_0^{(j)} = \frac{(-1)^{j-1}\Gamma(j-1/2)}{\pi}.
\end{align*}
As an immediate corollary for any $j\ge 1$ there exists a sequence $\{\gamma_k^{(j)}\}_{k=j}^\infty$  such that for any $m \ge j$ it holds
\begin{align}\label{eq:njasympdecomp}
    \frac{1}{n^{j-1/2}} = \sum_{k=j}^m \gamma_k^{(j)} a_n^{(k)} + O\left( \frac{1}{n^{m+1/2}}\right),\quad n\to \infty .
\end{align}
% \begin{corollary} \label{cor:basisschange} For any $j \ge 1$ there are sequences $\{\hat{\beta}_k^{(j)}\}_{k=j}^\infty,\{\gamma_k^{(j)}\}_{k=j}^\infty$ and $\{\hat{\gamma}_k^{(j)}\}_{k=j}^\infty$ such that for any $m \ge j$ it holds
% \begin{align*}
%     a_{n-1}^{(j)} &= \sum_{k=j}^m \frac{\hat{\beta}_k^{(j)}}{n^{k-1/2}} + O\left(\frac{1}{n^{m+1/2}}\right),\\
%     \frac{1}{n^{j-1/2}} &= \sum_{k=j}^m {\gamma_{k}^{(j)}}a_n^{(k)} + O\left( \frac{1}{n^{ m+1/2}}\right),\\
%     \frac{1}{n^{j-1/2}} &= \sum_{k=j}^m {\hat{\gamma}_{k}^{(j)}}a_{n-1}^{(k)} + O\left( \frac{1}{n^{m+1/2}}\right).
% \end{align*}
% \end{corollary}
% \textcolor{red}{(A nujni li nam vse tri razlojeniya? Na poslednee tochno ssilaemsya,)}
\begin{lemma}
\label{lem:an}
    The sequences $a_n^{(j)}$ for $j\ge 1$ posses the following properties:
    
$$ a_{n}^{(j)} - a_{n-1}^{(j)} = a_{n}^{(j+1)}, \quad n \ge 1; \eqno (i)$$

$$|a_{n-k}^{(j)} - a_n^{(j)}| \le \frac{c_j k}{ n^{j+1/2}}, \quad k \le n/2, \, n\ge 2. \eqno (ii) $$
\end{lemma}
\begin{proof}
% By differentiating \eqref{eq:anser} we obtain
% \begin{align*}
%     \sum_{n=1}^{\infty} na_n^{(j)}s^{n-1} = (3/2-j)(1-s)^{j-5/2} = (3/2-j)\sum_{n=0}^{\infty} a_n^{(j-1)}s^n, \quad j\ge 2.
% \end{align*}
% Hence it holds
% \begin{align} \label{eq:a_nrecrel}
%     a_n^{(j)} = \frac{(3/2 -j)}{n} a_{n-1}^{(j-1)}, \quad \quad j \ge 2, 
%     \;n\ge 1.
% \end{align}
% %For example due to \eqref{eq:a_n1asymp}
% %\begin{align*}
% %    a_n^{(2)} = \frac{-1/2}{n}a_{n-1}^{(1)} = -\frac{1}{2\sqrt{\pi} n^{3/2}} + O(n^{-5/2}).
% %\end{align*}
% Applying \eqref{eq:a_nrecrel} consecutively one has for arbitrary $j$ with $n\ge j-1$
% \begin{align}
%     a_n^{(j)} = (-1)^{j-1} \frac{(2j-3)!!}{2^{j-1}\sqrt{n}n(n-1)\dots(n-j+2)}a_{n-j+1}^{(1)}.
% \end{align}
% So using \eqref{eq:a_n1asymp}
% once more one can obtain the first property.

The statement (\textit{i}) easily follows:
\begin{align*}
    a_{n}^{(j)} - a_{n-1}^{(j)} = [s^n] \left( (1-s)  \sum_{n=0}^{\infty} a_n^{(j)}s^n\right) =& [s^n](1-s)(1-s)^{j-3/2}\\ =&[s^n] (1-s)^{j-1/2} = a_n^{(j+1)}, \quad n\ge 1.
\end{align*}
Next applying (\textit{i}) one obtains
\begin{equation*}
    a_{n-k}^{(j)} - a_n^{(j)} = -\sum_{\ell=0}^{k-1} a_{n-\ell}^{(j+1)}
\end{equation*}
and so it holds
\begin{align*}
    |a_{n-k}^{(j)} - a_n^{(j)}| \le \sum_{\ell=0}^{k-1} |a_{n-\ell}^{(j+1)}|.
\end{align*}
Due to \eqref{eq:anasympdecomp} we have $|a_n^{(j+1)}| \le C_j n^{-j-1/2}, \; n \ge 1$ for some $C_j$. Hence
\begin{align*}
    |a_{n-k}^{(j)} - a_n^{(j)}| \le C_j \sum_{\ell=0}^{k-1} (n-\ell)^{-j-1/2}% \le C_j k (n-k+1)^{-j-1/2}.
\end{align*}
Since $k\le n/2$ we have $(n-\ell)\ge n/2$, for $\ell < k$ and so we got
\begin{align*}
    |a_{n-k}^{(j)} - a_n^{(j)}| \le \frac{C_j}{2^{j+1/2}} k n^{-j-1/2}
\end{align*}
which is equivalent to (\textit{ii}) with $c_j = C_j/2^{j+1/2}$.
\end{proof}

\section{Asymptotic expansions in the central limit theorem}
\label{sec:clt_expan}
%%%%%%%%%%%%%%%%%%%%%%%%%%%%%%%%%%%%%%%%%%%%%%%%%%%%%%%%%%%%%%%%%%%%%%%%%%%
\subsection{Asymptotic expansions for $\pr(S_n\le x)$}
As we have mentioned in the introduction, our approach to asymptotic expansions for conditioned random walks requires detailed information on the asymptotic behaviour of probabilities $\pr(S_n\le x)$ and $\pr(S_n=x)$. In this section we collect some asymptotic expansions for these probabilities. 
The proofs of these results can be found at the end of the paper in Section~\ref{prop8.and.9}. 

We start by deriving asymptotic expansions for $\pr(S_n\le 0)$ which will be in used the proof of Theorem~\ref{thm:tau_0}.
\begin{proposition}
\label{prop:S_n<0}
{Assume that $\e X_1 = 0$, $ \e X_1^2 =: \sigma^2>0$} and that $\e|X_1|^{r+3}$ is finite for some integer {$r\ge0$}. Assume also that either the distribution of $X_1$ is lattice or $\limsup_{|t|\to\infty}\left|\e e^{itX_1}\right|<1$. Then there exist numbers 
$\theta_1,\theta_2,\ldots,\theta_{\lfloor r/2\rfloor + 1}$ such that
$$
\frac{\Delta_n}{n}
=\frac{1}{n}\left(\frac{1}{2}-\pr(S_n\le 0)\right)
=
\sum_{j=1}^{\lfloor r/2\rfloor + 1}\theta_j
a_{{n-1}}^{(j+1)}
+o\left( \frac{1}{n^{(r+3)/2}} \right).
$$
\end{proposition}

For our approach to asymptotic expansions for 
local probabilities $\pr(S_n\in dx,\tau_0>n)$ we need asymptotic expansions for local probabilities (densities in the absolute continuous case) of unconditioned walks.
\begin{proposition}
\label{prop:asymplattice}
Assume that $\e X_1 = 0$, $ \e X_1^2 =: \sigma^2 > 0$ and $\e|X_1|^{r+3}$ is finite for some integer $r\ge0$. 
Assume also that the distribution of $X_1$ is either lattice with maximal span $1$ or
absolutely continuous with bounded density. Set
$$p_n(x) = 
\begin{cases}
\pr (S_n = x), \text{ if the distribution of } X_1 \text{ is lattice},\\
\frac{d}{dx}\pr (S_n \in (0,x)) , \text{ if the distribution of } X_1 \text{ is absolutely continuous.} 
\end{cases}
$$
Then there exist polynomials $\theta_0, \theta_1,\ldots \theta_{\lf 
r/2\rf}$ ($\theta_k$ is of degree $2k$) such that, for some $C_r$,
\begin{align} \label{eq:prop6}
p_n(x)
=
\sum_{j=0}^{\lf 
r/2\rf}\theta_j(x) a_{n-1}^{(j+1)}
+
h_n(x)
\end{align}
and
\begin{align*}
    |h_n(x)|\le C_r \frac{(|x|+1)^{r+1}}{n^{(r+2)/2}}, \quad n\ge 1.
\end{align*}
\end{proposition}

%%%%%%%%%%%%%%%%%%%%%%%%%%%%%%%%%%%%%%%%%%%%%%%%%%%%%%%%%%%%%%%%%%%%%%%%%%%
%%%%%%%%%%%%%%%%%%%%%%%%%%%%%%%%%%%%%%%%%%%%%%%%%%%%%%%%%%%%%%%%%%%%%%%%%%%
\section{Special algebra of generating functions}
\label{sec:algebra}
Let $\{h_n\}$ be a sequence of real numbers.
Let $H(s)$ denote its generating function, i.e.
$$
H(s)=\sum_{n=0}^\infty h_n s^n.
$$
By $[s^n]H(s)$ we shall denote the 
$n$-th coefficient of a generating function.

In this section we shall define and study special classes of generating functions, which will play a very important role in our approach to asymptotic expansions for conditioned random walks.
\begin{definition}
\label{def.1}
Fix some $r\in\N_0$ and $m\in\N_0 \cup \{-1\}$.
Let $\mathcal{H}_m^r$ denote the class of generating functions \begin{equation*}
H(s) = \sum\limits_{n=0}^\infty h_n s^n, \ s \in [0,1)
\end{equation*}
of sequences $\{h_n\}_{n=0}^\infty$ satisfying the following two conditions:
\begin{enumerate}
    \item $h_n = O\left( \frac{(\log n)^r}{n^{(m+3)/2}}\right)$;
    \item $\sum\limits_{n=0}^\infty h_n P(n) = 0$ for every polynomial $P$ such that 
    its degree does not exceed $\lfloor m/2 \rfloor$.
\end{enumerate}
\end{definition}
Note that all sums $\sum\limits_{n=0}^\infty h_n P(n)$ exist since, by the first restriction on $\{h_n\}$,
\begin{equation*}
n^{\frac{m}{2}}h_n = O\left( \frac{(\log n)^r}{n^{3/2}}\right).
\end{equation*}
For every $H\in\mathcal{H}_m^r$ we define
\begin{equation*}
|H(s)|_{\mathcal{H}_m^r} 
=
\inf \left\{ C \ \Big| \ |h_0|, |h_1| \le C, \; \;
|h_n| \le \frac{C\log^r n}{n^{(m+3)/2}}
\quad\text{for all }  n\ge2  \right\}.
\end{equation*}
Obviously 
\begin{align*}
    |\lambda H(s)|_{\mathcal{H}_m^r} = \lambda| H(s)|_{\mathcal{H}_m^r}
\end{align*}
and 
$|\cdot|_{\mathcal{H}_m^r}$ is zero only for $H(s)\equiv0$. Also for $H,G \in \cH_m^r$ it holds 
\begin{equation*}
    |H + G|_{\cH_m^r} \le |H|_{\cH_m^r} + |G|_{\cH_m^r}.
\end{equation*}
Therefore, $|\cdot|_{\mathcal{H}_m^r}$ defines a norm on $\mathcal{H}_m^r$. It is also easy to see that $\cH_m^r$ is closed under summation and that 
$\cH_{m_0}^r \subset \cH_{m_1}^0$ for all $m_0 > m_1$ and all $r\ge0$.

Later for the functions used in the definition of the norm we will need the following observation.
\begin{remark} \label{rem:norm}
For every pair $(m,r)$ there exists a constant $L(m,r) \in (1, \infty)$ such that 
\begin{align*}
    \max_{k > \lf n/2\rf } \frac{\log^r k }{k^{(m+3)/2}} \le L(m,r) \frac{\log^r n}{n^{(m+3)/2}}, \quad n \ge 2.
\end{align*}
\end{remark}
\begin{proof}
Indeed, the function $\log^r k /k^{(m+3)/2}$ is monotone decreasing for 
$k\ge e^{2r/(m+3)}$. Therefore, for $k\ge 2e^{2r/(m+3)}$,
\begin{align*}
\max_{k > \lf n/2 \rf} \frac{\log^r k }{k^{(m+3)/2}} 
\le \frac{\log^r (n/2)) }{(n/2)^{(m+3)/2}} 
\le 2^{(m+3)/2}\frac{\log^r n }{n^{(m+3)/2}}.
\end{align*}
this implies immediately the existence of an appropriate constant $L(m,r)$.
\end{proof}

For our purposes 
one of the most important properties of the classes $\cH_m^r$ is the behaviour of its members under differentiation and integration. This property is stated in the following lemma.

\begin{lemma}
\label{dlemma}
For every $m \ge 1$ one the two following conditions are equivalent:
\begin{enumerate}
    \item $H \in \cH_m^r,$
    \medskip
    \item $\dH \in \cH_{m-2}^r \text{ and } \lim_{s \to 1} H(s) = 0.$
\end{enumerate}
Moreover, there exists a constant $C$ depending only on $m$ and $r$ such that 
$$
\left|\dH\right|_{\mathcal{H}_{m-2}^r}\le C(m,r)|H|_{\mathcal{H}_m^r}
$$
% \begin{equation*}
%     H \in \cH_m^r \Longleftrightarrow \dH \in \cH_{m-2}^r \text{ and } \lim_{s \to 1} H(s) = 0.
% \end{equation*}
\end{lemma}
\begin{proof}
Noting that $[s^n]\dH(s) = (n+1)[s^{n+1}]H(s)$, we infer that
\begin{equation*}
[s^n]H(s) = O\left( \frac{(\log n)^{r}}{n^{(m+3)/2}}\right) \Longleftrightarrow
[s^n]\frac{dH}{ds}(s) = O\left( \frac{(\log n)^{r}}{n^{(m+1)/2}}\right).
\end{equation*}
Thus, it remains to show equivalence of ``orthogonality to polynomials'' properties. We first notice that 
\begin{equation}
\label{eq:zerosum}
    \sum_{n=0}^\infty h_n = 0.
\end{equation}
in both cases.
Indeed, if  $H \in \cH_m^r$ for $m\ge1$ then \eqref{eq:zerosum} follows from the definition of $\cH_m^r$. If $\dH \in \cH_{m-2}^r$ then $(n+1)h_{n+1} = O\left((\log n)^{r}/n\right)$. In particular, $h_n = O\left(n^{-3/2}\right)$ and, consequently, $\sum_{n=0}^\infty |h_n|$ is finite. The assumption $\lim\limits_{s \to 1} H(s) = 0 $ yields then \eqref{eq:zerosum}.
Moreover, combining \eqref{eq:zerosum} and $h_n=O(n^{-3/2})$, we infer that 
$\lim_{s\to1}H(s)=0$ for every $h\in\cH_{m}^r$ with $m\ge1$.

Now suppose that $\dH \in \cH_{m-2}^r$. Hence $[s^n]\dH = O\left( \frac{(\log n)^{r}}{n^{(m+1)/2}}\right)$ and for every $k = 0, 1, \dots, \lfloor \frac{m-2}{2} \rfloor$ it holds
\begin{align}
\label{eq:HH'}
\nonumber
0 = \sum_{n=0}^\infty n^k [s^n]\dH(s)
&=\sum_{n=0}^\infty n^k(n+1)[s^{n+1}]H(s)
=\sum_{n=1}^\infty (n-1)^knh_n\\
&=\sum_{n=0}^\infty (n-1)^knh_n.
\end{align}
Combining \eqref{eq:zerosum} and \eqref{eq:HH'} we conclude that the sequence $\{h_n\}$ is orthogonal to all polynomials of degree $\le \lfloor m/2 \rfloor$. In the case $H \in \cH^r_m$ we can read \eqref{eq:HH'} in opposite direction starting with orthogonality of $h_n$ to $(n-1)^kn$ for $k=0, 1, \dots ,\lfloor \frac{m-2}{2} \rfloor$. Thus we have  obtained orthogonality of $[s^n]\dH$ to $n^k, k =0, 1, \dots, \lfloor \frac{m-2}{2} \rfloor$ Then  $\dH\in\cH_{m-2}^r$ and the proof is complete.

The relation between the norms of $H$ and $\dH$ is obvious.
\end{proof}
% If we assume now that $H$ belongs to $\cH_m^r$ then the right hand side in \eqref{eq:HH'} equals zero for each
% $k\le \lfloor \frac{m}{2} \rfloor -1=\lfloor \frac{m-2}{2} \rfloor$.
% Consequently, $\sum_{n=0}^\infty n^k [s^n]H'(s)=0$ for every
% $k\le\lfloor \frac{m-2}{2} \rfloor$. Taking into account \eqref{eq:zerosum}, we conclude that $H'(s)\in\cH_{m-2}^r$. 

% If we assume now that $H'(s)\in\cH_{m-2}^r$ then, using once again \eqref{eq:HH'}, we obtain 
% $$
% \sum_{n=0}^\infty (n-1)^kn[s^n]H(s)=0,
% \quad k=0,1,\ldots, \lfloor \frac{m-2}{2} \rfloor.
% $$
% Noting that 
% $\{1,n,n(n-1),\ldots,n(n-1)^{\lfloor \frac{m-2}{2} \rfloor}\}$ is a basis
% in the space of polynomials of the degree $\lfloor \frac{m}{2} \rfloor$, we infer 'orthogonality' to any polynomial of this degree. Thus, $H(s)\in\cH_m^r$ and the proof is complete.

\begin{lemma} \label{mlemma}
If $H, G\in \cH_{m}^{r}$ for some $m\ge 0$ then

\begin{equation*}
    H \cdot G \in \cH^{r}_{m}.
\end{equation*}
Moreover, there exists a constant $C(m,r)$ such that 
\begin{equation} \label{1410}
|H \cdot G|_{\cH_m^r} \le 
C(m,r) |H|_{\cH_m^r}|G|_{\cH_m^r}.
\end{equation}
\end{lemma}

\begin{proof}
Set $ M^H_\varepsilon=|H|_{\cH_m^r}+\varepsilon$,
$M^G_\varepsilon=|G|_{\cH_m^r}+\varepsilon$ for some $\varepsilon > 0$ and define
$$
u_n:=[s^n](HG)(s)=\sum_{k=0}^n h_k g_{n-k},\quad n\ge0.
$$
Obviously, 
$|u_0| = |h_0|\cdot|g_0| \le  M^H_\varepsilon M^G_\varepsilon$ and
$|u_1| = |h_0 g_1 + h_1 g_0| \le 2 M^H_\varepsilon M^G_\varepsilon$.
Furthermore, by the definition of the norm $|\cdot|_{\cH_m^r}$ and due to Remark \ref{rem:norm}, for $n\ge 2$ one has
\begin{align*}
|u_n|
&\le \sum_{k=0}^n |h_k| |g_{n-k}| 
    =\sum_{k=0}^{\lfloor n/2\rfloor } |h_k| |g_{n-k}| 
    +\sum_{k = \lfloor n/2\rfloor + 1}^{n} |h_{k}| |g_{n-k}|\\ 
&\le  \left(\max_{k > \lf n/2 \rf} \frac{\log^r k}{k^{(m+3)/2}} \right)
    \left( M^G_\varepsilon\sum_{k=0}^{\lfloor n/2\rfloor } |h_{k}| + M^H_\varepsilon \sum_{k = 0}^{n - \lfloor n/2\rfloor} |g_{k}|\right)
    \\
&\le  L(m,r) \frac{\log^r n}{n^{(m+3)/2}} 
    \left( M^G_\varepsilon\sum_{k=0}^{\infty} |h_{k}| + M^H_\varepsilon\sum_{k = 0}^{\infty} |g_{k}|\right).
\end{align*}
Furthermore,
\begin{align*}
M^G_\varepsilon\sum_{k=0}^{\infty} |h_{k}|
+M^H_\varepsilon \sum_{k = 0}^{\infty} |g_{k}| 
    \le  2M^H_\varepsilon M^G_\varepsilon
    \left(2+ \sum_{n=2}^\infty
\frac{\log^r n}{n^{(m+3)/2}}\right).
\end{align*}
This leads to the bound
\begin{align}
\label{eq:n>=2}
|u_n| \le 2 M^H_\varepsilon M^G_\varepsilon L(m,r) \frac{\log^r n}{n^{(m+3)/2}} \left(2+ \sum_{n=2}^\infty
\frac{\log^r n}{n^{(m+3)/2}}\right), \quad n\ge 2.
\end{align}
Combining \eqref{eq:n>=2} with the bounds
$$
|u_0|\le M^H_\varepsilon M^G_\varepsilon,\quad 
|u_1|  \le 2 M^H_\varepsilon M^G_\varepsilon
$$
and letting $\varepsilon \to 0$ we obtain \eqref{1410} with 
$$
C(m,r):=2 L(m,r)\left( 2+ \sum_{n=2}^\infty
\frac{\log^r n}{n^{(m+3)/2}}\right).
$$

To finish the proof that $(HG)(s)$  belongs to $\cH_m^r$ we are left to show the "orthogonality to polynomials" property.
We have already shown that the asymptotically $u_n$ is $O\left(\frac{\log^r n}{n^{(m+3)/2}}\right)$, and, in particular, $\sum_{n=0}^\infty u_n$
is finite. Therefore,
\begin{align} \label{eq:u_nzerosum}
\sum_{n=0}^\infty u_n=\lim_{s\to1}(HG)(s)
=\lim_{s\to1}H(s)\lim_{s\to1}G(s)=0,
\end{align}
where the latter equality follows from the assumption that 
$G(s),H(s)\in\cH_m^r$ with $m\ge0$. Thus, the proof is finished for 
$m=0,1$. 

In the case $m\ge2$ we shall use the induction. Assume that the lemma is valid for $\cH_k^r$ with $k\le m-1$. Then, by Lemma~\ref{dlemma} and by the induction assumption,
\begin{equation*}
    H(s)\dG(s),\ G(s)\dH(s)\in \cH_{m-2}^r.
\end{equation*}
Therefore,
\begin{equation*}
    \frac{d(HG)}{d s}(s) = H(s)\dG(s) +  G(s)\dH(s)\in \cH_{m-2}^r.
\end{equation*}
Finally, due to \eqref{eq:u_nzerosum} we can apply Lemma \ref{dlemma} once again, we conclude that
\begin{equation*}
H(s)G(s) \in \cH_m^r.
\end{equation*}
Thus, the proof is finished.
\end{proof}
\begin{corollary}\label{mult}
Assume that $H \in \cH_{m}^r$ with some $m\ge 0$. Then, for every $k\ge2$, $H^k \in \cH_m^r$ and
\begin{equation*}
    |H^k(s)|_{\cH_m^r} \le C^{k-1}(m,r) \left(|H(s)|_{\cH_m^r}\right)^{k}.
\end{equation*}
\end{corollary}
\begin{lemma} \label{explemma}
Assume that $H \in \cH_m^r$ with some $m\ge 0$. For every sequence
$\{p_k\}$,
satisfying
\begin{equation}
\label{eq:p_assump} 
\sum_{k=1}^\infty |p_k|R^k<\infty
\quad\text{for some}\quad R>C(m,r)|H(s)|_{\cH_m^r},
\end{equation}
one has
\begin{equation*}
\sum_{k=1}^\infty p_k H^k(s) \in \cH_m^r.
\end{equation*}
\end{lemma}
\begin{proof}
Let us first show that 
\begin{equation}
\label{changecondition}
\sum_{k=1}^\infty |p_k| \sum_{n=0}^\infty |h_n^{(k)}| n^{m/2} < \infty,
\end{equation}
where
\begin{equation*}
h_n^{(k)} := [s^n] H^k(s),\quad n\ge0,\ k\ge1.
\end{equation*}
Fix some $\varepsilon>0$ such that $M:=|H(s)|_{\cH_m^r}+\varepsilon < R$.
By Corollary \ref{mult}, for every $k\ge1$ one has the estimates
\begin{align}\label{eq:hnk_bound1}
|h_0^{(k)}|, |h_1^{(k)}| \le C^{k-1}(m,r)M^k
\end{align}
and
\begin{align}
\label{eq:hnk_bound2}
|h_n^{(k)}| \le C^{k-1}(m,r)M^k 
\frac{\log^r n}{n^{(m+3)/2}},\ n\ge 2.
\end{align}
Hence
\begin{align*}
\sum_{n=0}^\infty |h_n^{(k)}| n^{m/2} 
&\le {C^{k-1}(m,k)}M^k \left(2+
\sum_{n=2}^\infty \frac{\log^r n}{n^{(m+3)/2}}n^{m/2}
\right)\\
&=  {C^{k-1}(m,r)}M^k \left( 2 + \sum_{n=2}^\infty \frac{\log^r n}{n^{3/2}} \right).
\end{align*}
Letting $A_r := 2 + \sum_{n=2}^\infty \log^r n/ n^{3/2}$, we have
\begin{align*}
\sum_{k=1}^\infty |p_k| \sum_{n=0}^\infty |h_n^{(k)}| n^{m/2} 
&\le \frac{A_r}{C(m,r)}\sum_{k=1}^\infty |p_k| (C(m,r)M)^k\\
&\le \frac{A_r}{C(m,r)}\sum_{k=1}^\infty |p_k| R^k < \infty.
\end{align*}
Thus, \eqref{changecondition} is valid. In particular we have 
\begin{equation*}
\sum_{k=1}^\infty |p_k| \sum_{n=0}^\infty |h_n^{(k)}| < \infty.
\end{equation*}
This leads, by applying the Fubini theorem, to the equality
\begin{equation*}
    \sum_{k=1}^\infty p_k H^k(s) 
    =
    \sum_{n=0}^\infty s^n 
    \left( \sum_{k=1}^\infty p_k h_n^{(k)}\right).
\end{equation*}
Combining \eqref{eq:p_assump} and \eqref{eq:hnk_bound2}, we get, for $n\ge 2$,
\begin{align*}
\left|\sum_{k=1}^\infty p_k h_n^{(k)}\right| 
&\le\frac{\log^r n}{n^{(m+3)/2}} 
\sum_{k=1}^\infty {C^{k-1}(m,k)}M^k|p_k| 
\\ &\le \frac{\sum_{k=1}^\infty|p_k|R^k}{C(m,r)} \cdot \frac{\log^r n}{n^{(m+3)/2}} = O\left( \frac{\log^r n}{n^{(m+3)/2}}\right).
\end{align*}
Thus, the condition (i) in Definition~\ref{def.1} is satisfied.

Next, by Corollary~\ref{mult}, $H^k(s)\in\cH_m^r$ for every $k\ge1$.
Therefore, for every polynomial $P(n)$ of degree
$\le \lfloor m/2\rfloor$ we have
\begin{equation*}
\sum_{n=0}^\infty h_n^{(k)} P(n) = 0,\quad\text{for all } k\ge1.
\end{equation*}
Then, in view of \eqref{changecondition},
\begin{equation*}
    \sum_{n=0}^\infty P(n) \left( \sum_{k=1}^\infty p_k h_n^{(k)}\right) 
    =
    \sum_{k=1}^\infty p_k   \left(\sum_{n=0}^\infty P(n)h_n^{(k)}\right) =0.
\end{equation*}
This completes the proof of the lemma.
\end{proof}

Next we prove some results concerning the functions $(1-s)^m$, which shall play an important role in the asymptotic expansions for conditioned random walks.
\begin{lemma}
\label{1-slemma}
For every $k\ge1$ one has
$$
(1-s)^{k-1/2} \in \cH^0_{2k-2}.
$$
\end{lemma}

\begin{proof}
In the case $k=1$, by definition of $a_n^{(2)}$ and by \eqref{eq:anasympdecomp} one has
\begin{align*}
    a^{(2)}_n=[s^n](1-s)^{1/2}=O(n^{-3/2})
\end{align*}
and
$$
\sum_{n=0}^\infty a^{(2)}_n=\lim_{s\to1}(1-s)^{1/2}=0.
$$
Thus, $(1-s)^{1/2}\in \cH_{0}^0$.
Noting that
$$
\frac{d}{ds}(1-s)^{k-1/2}=(k-1/2)(1-s)^{(k-1)-1/2}
$$
and applying Lemma~\ref{dlemma}, one gets the desired result
by induction.
\end{proof}
\begin{lemma}
\label{lem:hlemma_1}
If $H\in \cH_m^r$ with some $m \ge 1$ then
$$
\frac{H(s)}{(1-s)} \in \cH_{m-2}^r
$$    
and
$$
\left|\frac{H(s)}{1-s}\right|_{\cH_{m-2}^r}
\le C\left|H(s)\right|_{\cH_m^r}
$$
for some constant $C$ which depends on $m$ and $r$ only.
\end{lemma}

\begin{proof}
By the definition of $\cH_m^r$, $\sum_{n=0}^\infty h_n=0$
and
$$
|h_0|\le M, |h_1|\le M\quad\text{and}\quad
|h_n|\le M\frac{\log^r n}{n^{(m+3)/2}},\quad n\ge2
$$
for every constant $M>\left|H(s)\right|_{\cH_m^r}$.
Then
\begin{align*}
\left|[s^0] \frac{H(s)}{(1-s)}\right|=|h_0|\le M
\end{align*}
and, for every $n\ge1$,
\begin{align*}
\left|[s^n] \frac{H(s)}{(1-s)}\right|
&=\left|\sum_{k=0}^n h_k \right|
= \left|-\sum_{k=n+1}^\infty h_k\right|\\
&\le M\sum_{k=n+1}^\infty\frac{\log^r k}{k^{(m+3)/2}}
\le C M\frac{\log^{r} n}{n^{((m-2)+3)/2}}.
\end{align*}
This gives already the desired relation for the norms.
To complete the proof, it remains to show that 
\begin{equation} \label{1646}
    \sum_{n=0}^\infty \left(\sum_{k=0}^n h_n\right) P(n) = 0
\end{equation}
for every polynomial $P$ of degree $\le \lfloor \frac{m-2}{2}\rfloor$.

Changing the order of summation, we get 
\begin{align*}
\sum_{n=0}^\infty\left(\sum_{k=0}^n h_k\right)P(n)
=-\sum_{n=0}^\infty\left(\sum_{k=n+1}^\infty h_k\right)P(n)
=-\sum_{k=1}^\infty h_k\sum_{n=0}^{k-1}P(n) = 0.
\end{align*}
The last equality holds since $Q(k):=\sum_{n=0}^{k-1}P(n)$ is a polynomial of order
$\le \lfloor \frac{m}{2} \rfloor$ and $h_n$ is orthogonal to it. Hence we obtain \eqref{1646}.
\end{proof} 
\begin{lemma}
\label{lem:hlemma_2}
If $H(s)\in\cH_m^r$ for some $m\ge0$ then 
$$
H(s)(1-s)^{k/2}\in\cH_m^r\quad\text{for every }k\ge1.
$$
Moreover, there exists $C$ independent of $H$ such that 
$$
\left|H(s)(1-s)^{k/2}\right|_{\cH_m^r}
\le C\left|H(s)\right|_{\cH_m^r}
$$
\end{lemma}
\begin{proof}
Obviously, it suffices to consider only the case $k=1$.

Lemma \ref{1-slemma} says that $(1-s)^{1/2} \in \cH^r_0, \; k\ge 1$. If $m=0$ then, by Lemma \ref{mlemma}, $H(s)(1-s)^{1/2}$ also belongs to $\cH_0^r$ and 
$\left|H(s)(1-s)^{1/2}\right|_{\cH_0^r}
\le C\left|H(s)\right|_{\cH_0^r}.$
% the desired result follows from Lemmata~\ref{1-slemma}
% and~\ref{mlemma}.

Consider now the case $m=1$. We first notice that 
\begin{align*}
[s^n]H(s)(1-s)^{1/2}
&=\sum_{k=0}^n h_k a^{(2)}_{n-k}\\
&=\sum_{k=0}^{\lfloor n/2 \rfloor } h_k a^{(2)}_{n-k}
+\sum_{k=\lfloor n/2 \rfloor + 1}^n h_k a^{(2)}_{n-k}.
\end{align*}
For the second sum we have 
\begin{align*}
\left|\sum_{k=\lfloor n/2 \rfloor +1}^n h_k a^{(2)}_{n-k}\right|
\le \max_{k\ge \lfloor n/2 \rfloor +1}|h_k| \sum_{k=0}^\infty |a^{(2)}_{k}|.
\end{align*}
Since $|a_k^{(2)}| = O(n^{-3/2})$, $\sum_{k=1}^\infty|a_k^{(2)}|<\infty$. Furthermore,
the assumption $H\in\cH_1^r$ implies that $|h_0|\le M$, $|h_1|\le M$ and $|h_k|\le M\frac{\log^rk}{k^2}$, $k\ge 2$ for every $M>|H|_{\cH_1^r}$. As a result,
\begin{align}
\label{eq:lem10.1}
\left|\sum_{k=\lfloor n/2 \rfloor +1}^n h_k a^{(2)}_{n-k}\right| \le C M\frac{\log^r n}{n^2}.
\end{align}
The assumption $H\in\cH_1^r$ implies also that $\sum_{k=0}^\infty h_k=0$.
Using this equality, we get
\begin{align*}
\sum_{k=0}^{\lfloor n/2 \rfloor} h_k a^{(2)}_{n-k}
&=a^{(2)}_{n}\sum_{k=0}^{\lfloor n/2 \rfloor} h_k
+\sum_{k=0}^{\lfloor n/2 \rfloor} h_k \left(a^{(2)}_{n-k}-a^{(2)}_{n}\right)\\
&=-a^{(2)}_{n}\sum_{k=\lfloor n/2 \rfloor + 1}^\infty h_k
+\sum_{k=0}^{\lfloor n/2 \rfloor} h_k \left(a^{(2)}_{n-k}-a^{(2)}_{n}\right).
\end{align*}
The first term on the right hand side is also bounded by  
$C M\frac{\log^r n}{n^{5/2}}$ since $a_n^{(2)} = O(n^{-3/2})$ and $|h_n| \le M\frac{\log^r n}{n^{2}}$.
By the property (\textit{ii}) from Lemma \ref{lem:an}, for $1\le k\le n/2$ one has the bound
$$
|a^{(2)}_{n-k}-a^{(2)}_{n}|\le c_2 k n^{-5/2}.
$$
Hence
\begin{align*}
\left|\sum_{k=0}^{\lfloor n/2 \rfloor} h_k \left(a^{(2)}_{n-k}-a^{(2)}_{n}\right) \right|
\le c_2 n^{-5/2}\sum_{k=0}^{\lfloor n/2 \rfloor} k|h_k| 
\le CM\frac{\log^{r+1}n}{n^{5/2}}.
\end{align*}
So finally
\begin{equation}
\label{eq:lem10.2}
\left|\sum_{k=0}^{\lfloor n/2 \rfloor} h_k a^{(2)}_{n-k}\right|
\le CM\frac{\log^{r+1} n}{n^{5/2}}.
\end{equation}
Combining \eqref{eq:lem10.1} and \eqref{eq:lem10.2}, we have
$$ 
\left|[s^n]H(s)(1-s)^{1/2}\right|
\le CM\frac{\log^r n}{n^{2}},\quad n\ge 2.
$$
Combining this estimate with $\lim_{s\to 1}H(s)(1-s)^{1/2}=0$, we conclude that $H(s)(1-s)^{1/2}\in\cH_1^r$
and obtain also the inequality for the norms.

\medskip
The cases $m=0$ and $m=1$ can now be used as the basis for the following
induction argument. By Lemmas~\ref{dlemma} and \ref{lem:hlemma_1},
$\dH$ and $H(s)/(1-s)$ belong to $\cH_{m-2}^r$. Then, by the induction assumption,
\begin{align*}
    \dH(s)(1-s)^{1/2}-\frac{1}{2}\frac{H(s)}{(1-s)}(1-s)^{1/2}\in\cH_{m-2}^r.
\end{align*}
Furthermore,
\begin{equation}
\label{eq:norm-ineq}
\left|\dH(s)(1-s)^{1/2}-\frac{1}{2}\frac{H(s)}{(1-s)}(1-s)^{1/2}\right|_{\cH_{m-2}^r}
\le C\left|H(s)\right|_{\cH_m^r}.
\end{equation}
Noting that the expression on the left hand side equals $\frac{d}{ds}\left( H(s)(1-s)^{1/2}\right)$, we have
\begin{equation*} 
\frac{d}{ds}\left( H(s)(1-s)^{1/2}\right) \in\cH_{m-2}^r.
\end{equation*}
Since $\lim\limits_{s\to1} H(s) = 0$, we may apply Lemma~\ref{dlemma} to conclude that $H(s)(1-s)^{1/2}$ also belongs to the class $\cH_m^r$. Thus, it remains to show the inequality for the norms. First, we notice that, for all $n\ge1$, 
\begin{align*}
n[s^n]H(s)(1-s)^{1/2}
&=[s^{n-1}]\frac{d}{ds}\left( H(s)(1-s)^{1/2}\right)\\
&=[s^{n-1}]\left(\dH(s)(1-s)^{1/2}-\frac{1}{2}\frac{H(s)}{(1-s)}(1-s)^{1/2}\right).
\end{align*}
From this equality and from \eqref{eq:norm-ineq} we infer that
\begin{align*}
\left|[s^n]H(s)(1-s)^{1/2}\right|
&\le \left(C\left|H(s)\right|_{\cH_m^r}+\varepsilon\right)
\frac{\log^r n}{n(n-1)^{(m+1)/2}}\\
&\le (3/2)^{(m+1)/2}
\left(C\left|H(s)\right|_{\cH_m^r}+\varepsilon\right)
\frac{\log^r n}{n^{(m+3)/2}},\quad n\ge 3.
\end{align*}
Finally,
\begin{align*}
&\left|[s^0]H(s)(1-s)^{1/2}\right|=|h_0|\le \left|H(s)\right|_{\cH_m^r}+\varepsilon,\\
&\left|[s^1]H(s)(1-s)^{1/2}\right|
=|h_0a_1^{(2)}+h_1|\le 2(\left|H(s)\right|_{\cH_m^r}+\varepsilon),\\
&\left|[s^1]H(s)(1-s)^{1/2}\right|
=|h_0a_2^{(2)}+h_1a_1^{(2)}+h_2|\le 3(\left|H(s)\right|_{\cH_m^r}+\varepsilon).
\end{align*}
Combining these estimates and letting $\varepsilon\to0$, we obtain the desired bound for the norm.
\end{proof}
Next lemma is the crucial part of the theory. For all earlier examples  any multiplications keep parameter $m$ the same, but it happens that under multiplication by $(1-s)^{-1/2}$ parameter $m$ decreases by $1$. At the same time the function $(1-s)^{-1/2}$ arises naturally since $[s^n](1-s)^{-1/2} \sim \frac{1}{\sqrt{\pi n}}$ so we have to deal with this function.

\begin{lemma}
\label{lem:hlemma_3}
If $H(s)\in\cH_m^r$ with some $m\ge0$ then 
$$
\frac{H(s)}{\sqrt{1-s}}\in \cH_{m-1}^{r+\delta_m},
$$
where $\delta_m = \mathds{1}_{\{m \text{ is odd}\}}$.
Moreover, there exists a constant $C$ depending only on $m$ and $r$ such that
$$
\left|[s^n]\frac{H(s)}{\sqrt{1-s}}\right|
\le C|H|_{\cH_m^r}\frac{\log^{r+\delta_m}n}{n^{(m+2)/2}},\quad n\ge2
$$
and 
$$
\left|\frac{H(s)}{\sqrt{1-s}}\right|_{\cH_m^r}
\le C |H|_{\cH_m^r}.
$$
% \begin{align*}
% \delta_m = \mathds{1}
% _{\{m\equiv 1 \mod 2\}}.
% % \delta_m = \begin{cases}
% %     1, \quad m \equiv 1 \text{(mod 2)},\\
% %     0, \quad m \equiv 0 \text{(mod 2)}.
% % \end{cases}
% \end{align*}
\end{lemma}

\begin{proof} We shall use the induction over $m$ again. The cases $m=0$ and $m=1$ will serve as a basis for our argument. First of all, we show that for $m=0,1$ one has
\begin{equation}
\label{1800}
\left|[s^n] \frac{H(s)}{\sqrt{1-s}}\right| 
= \left|\sum_{k=0}^n h_k a^{(1)}_{n-k}\right| 
\le C |H|_{\cH_m^r}\frac{\log^{r+\delta_m}n}{n^{(m+2)/2}}.
\end{equation}
Split the sum in \eqref{1800} into two parts:
\begin{equation}
\label{1731}
\sum_{k=0}^n h_k a^{(1)}_{n-k} 
=\sum_{k=0}^{\lfloor \frac{n}{2} \rfloor} h_k a^{(1)}_{n-k}
+\sum_{k=\lfloor \frac{n}{2} \rfloor + 1}^n h_k a^{(1)}_{n-k}.
\end{equation}
Due to the assumption $H(s)\in \cH_m^r$ and Remark \ref{rem:norm},
\begin{equation*}
\max_{k > \lf n/2 \rf }|h_k| 
\le |H(s)|_{\cH_m^r} L(m,r) \frac{\log^r n}{n^{(m+3)/2}}.
 \end{equation*}
Therefore,
\begin{align}
\label{eq:2sum}
\sum_{k=\lfloor \frac{n}{2} \rfloor + 1}^n |h_k| |a^{(1)}_{n-k}| 
&\le |H(s)|_{\cH_m^r}  L(m,r) \frac{\log^r n}{n^{(m+3)/2}}
\sum_{k=\lfloor \frac{n}{2} \rfloor + 1}^n |a^{(1)}_{n-k}|.
\end{align}
We know from \eqref{eq:anasympdecomp} that
\begin{align}\label{eq:a^1}
a_k^{(1)} \sim\frac{1}{\sqrt{\pi k}}.
\end{align}
Consequently,
\begin{align*}
    \sum_{k=\lfloor \frac{n}{2} \rfloor + 1}^n |a^{(1)}_{n-k}| = O(n^{1/2}).
\end{align*}
Combining this with \eqref{eq:2sum}, we conclude that
\begin{equation}
\label{eq:second_sum}
\sum_{k=\lfloor \frac{n}{2} \rfloor + 1}^n |h_k| |a^{(1)}_{n-k}| 
\le C|H|_{\cH_m^r}\frac{\log^r n}{n^{(m+2)/2}}.
\end{equation}
Using the equality 
$\sum\limits_{n=0}^\infty h_n = 0$, we obtain
\begin{align*}
\sum_{k=0}^{\lfloor \frac{n}{2} \rfloor} h_k a^{(1)}_{n-k} 
&=\sum_{k=0}^{\lfloor \frac{n}{2} \rfloor} h_k a^{(1)}_{n-k}
  -a^{(1)}_n\sum_{k=0}^\infty h_k \\
&=\sum_{k=1}^{\lfloor \frac{n}{2} \rfloor} h_k (a^{(1)}_{n-k}-a^{(1)}_n)
  -a^{(1)}_n\sum_{k=\lfloor \frac{n}{2} \rfloor + 1}^\infty h_k.
\end{align*}
The assumption $H\in\cH_m^r$ implies that 
\begin{equation*}
\left|\sum_{k=\lfloor \frac{n}{2} \rfloor + 1}^\infty h_k\right| 
\le C|H|_{\cH_m^r}\frac{\log^r n}{n^{(m+1)/2}}. 
\end{equation*}
Combining this with \eqref{eq:a^1}, we get
\begin{equation*}
\left|a^{(1)}_n \sum_{k=\lfloor \frac{n}{2} \rfloor + 1}^\infty h_k\right|
\le C|H|_{\cH_m^r}\frac{\log^r n}{n^{(m+2)/2}}.
\end{equation*}
From this bound and from \eqref{eq:second_sum} we infer that \eqref{1800} will be proven if we show that
\begin{equation}
\label{eq:remains}
\left|\sum_{k=1}^{\lfloor \frac{n}{2} \rfloor} h_k (a^{(1)}_{n-k}-a^{(1)}_n)\right|
\le C|H|_{\cH_m^r}\frac{\log^r n}{n^{(m+2)/2}}.
\end{equation}
In Lemma \ref{lem:an} we have shown that  
$$
|a^{(1)}_{n-k}-a^{(1)}_n|\le c_1\frac{k}{n^{3/2}},\quad k\le\frac{n}{2}.
$$
Therefore,
$$
\left|
\sum_{k=1}^{\lfloor \frac{n}{2} \rfloor} h_k (a^{(1)}_{n-k}-a^{(1)}_n)\right|
\le Cn^{-3/2}\sum_{k=1}^{\lfloor \frac{n}{2} \rfloor} k|h_k|.
$$
It remains to notice that 
\begin{align*}
\sum_{k=1}^{\lfloor \frac{n}{2} \rfloor} k|h_k|
\le C|H|_{\cH_m^r}\frac{\log^{r+\delta_m} n}{n^{(m-1)/2}}
\end{align*}
for $m=0$ and for $m=1$.
The value $\delta_m$ arises when we sum up the sequence of order $1/n$.
Thus, \eqref{1800} is proven.

\medskip
Now we turn to the "orthogonality" property. Here we have to consider the case $m=1$ only;
for $m=0$ there are no orthogonality restrictions in the definition of $\cH_{-1}^r$.
We show a bit more than needed, we prove that
\begin{equation} \label{1808}
    \sum_{n=0}^\infty [s^n] \frac{H(s)}{\sqrt{1-s}} = 0
\end{equation}
for every $H\in \cH_{m}^r$, $m\ge1$.
Fix some $N\ge 1$ and consider the partial sum\begin{align*}
\sum_{n=0}^N [s^n] \frac{H(s)}{\sqrt{1-s}}
&=\sum_{n=0}^N \sum_{k=0}^n a^{(1)}_k h_{n-k}\\
&=\sum_{k=0}^N \sum_{n=k}^N a^{(1)}_k h_{n-k} 
=\sum_{k=0}^N a^{(1)}_k  \left(\sum_{n=0}^{N-k} h_{n}\right).
\end{align*}
Set, for brevity,
$$
v_n := \sum\limits_{k=0}^n h_k.
$$
Since $H(s)\in\cH_1^r$ one has $\sum_{k=0}^\infty h_n=0$ and, consequently,
$$
v_n = -\sum\limits_{k=n+1}^\infty h_k.
$$
We know that $h_n = O\left(\frac{\log^r n}{n^2}\right)$. Consequently,  
$v_n = O\left(\frac{\log^r n}{n}\right)$. This bound implies that, for some constants $C_0$ and $C_1$,
\begin{align*}
\left|\sum_{k=0}^{\lfloor \frac{N}{2} \rfloor} a^{(1)}_k  v_{N-k}\right| 
&\le C_0\frac{\log^r N}{N}
\sum_{k=0}^{\lfloor \frac{N}{2} \rfloor}a^{(1)}_k\\
&\le C_1\frac{\log^r N}{N}\sqrt{N}\rightarrow0
\quad \text{as }N\to \infty,
\end{align*}
in the last step we have used again the relation
$a_k^{(1)}\sim\frac{1}{\sqrt{\pi n}}$.    
By the same fact, for some constants $C_2, C_3$,
\begin{align*}
\left|\sum_{k=\lfloor \frac{N}{2} \rfloor + 1}^N a^{(1)}_k v_{N-k}\right|
&\le\frac{C_2}{\sqrt{N}}\sum_{k=\lfloor\frac{N}{2}\rfloor + 1}^N|v_{N-k}|
\le \frac{C_2}{\sqrt{N}}\sum_{k=0}^{N}|v_{k}|
\\
&\le\frac{C_3}{\sqrt{N}}\left(2+\sum_{k=2}^N\frac{\log^r k}{k}\right) 
\le C_3 \frac{\log^{r+1}N}{\sqrt{N}}\rightarrow0
\quad \text{as }N\to \infty.
\end{align*}
Therefore, the sequence of partial sums 
$\sum_{n=0}^N[s^n]\frac{H(s)}{(1-s)^{1/2}}$ converges to zero, and
\eqref{1808} holds.

So we have proved the lemma for $m=0,1$. Now we apply the induction argument.
Fix some $m \ge 2$ and assume that the lemma is valid for all smaller indices. Clearly,
\begin{equation*}
\frac{d}{ds}\frac{H(s)}{\sqrt{1-s}} 
=\frac{\dH(s)}{\sqrt{1-s}}-
\frac{1}{2}\frac{H(s)}{1-s}\frac{1}{\sqrt{1-s}}.
\end{equation*}
We already have shown in Lemmata \ref{dlemma} and \ref{lem:hlemma_1} that if $H\in \cH_m^r$ then
\begin{equation*}
    \frac{H(s)}{1-s}, \dH(s) \in \cH_{m-2}^r.
\end{equation*}
By the induction assumption,
\begin{equation*}
\frac{\dH(s)}{\sqrt{1-s}}, 
\frac{H(s)}{1-s}\frac{1}{\sqrt{1-s}} \in 
\cH_{m-3}^{r+\delta_m}.
\end{equation*}
Consequently,
\begin{equation*}
\frac{d}{ds}\frac{H(s)}{\sqrt{1-s}}\in \cH_{m-3}^{r+\delta_m}.
\end{equation*}
Since we have proved \eqref{1808} we have $\lim\limits_{s \to 1}H(s)/\sqrt{1-s} = 0$ and we may apply Lemma~\ref{dlemma} to get
\begin{equation*}
\frac{H(s)}{\sqrt{1-s}} \in \cH_{m-1}^{r+\delta_m}.
\end{equation*}
\end{proof}
\begin{lemma}
\label{lem:partial}
Assume that $H(s)=\sum_{k=0}^\infty h_ks^k$ belongs to $\cH_m^r$ with some $m\ge0$.
Then, for every $j\ge1$ there exists a constant $A_j$ depending on $j,m$ and $r$ only such that
$$
\sum_{k=0}^{N_n} a_{n-k}^{(j)}h_k
\le A_j|H(s)|_{\cH_m^r}\frac{(\log n)^{r+\delta_m}}{n^{(m+2j)/2}}
$$
for all $N_n\in\left[\frac{n}{3},\frac{2n}{3}\right]$.
\end{lemma}
\begin{proof}
Set
$$
v_k:=\sum_{l=0}^k h_l,\quad k\ge0.
$$
Then one has 
\begin{align*}
\sum_{k=0}^{N_n} a_{n-k}^{(j)}h_k
&=a_n^{(j)}v_0+\sum_{k=1}^{N_n} a_{n-k}^{(j)}(v_k-v_{k-1})\\
&=(a_n^{(j)}-a_{n-1}^{(j)})v_0+\sum_{k=1}^{N_n} a_{n-k}^{(j)}v_k
-\sum_{k=2}^{N_n} a_{n-k}^{(j)}v_{k-1}\\
&=(a_n^{(j)}-a_{n-1}^{(j)})v_0+\sum_{k=1}^{N_n-1}(a_{n-k}^{(j)}-a_{n-k-1}^{(j)})v_k
+a_{n-N_n}^{(j)}v_{N_n}\\
&=\sum_{k=0}^{N_n-1}(a_{n-k}^{(j)}-a_{n-k-1}^{(j)})v_k
+a_{n-N_n}^{(j)}v_{N_n}.
\end{align*}
We know from Lemma \ref{lem:an} that $a_n^{(j)}-a_{n-1}^{(j)}=a_n^{(j+1)}$. Therefore,
\begin{equation}
\label{eq:partial.1}
\sum_{k=0}^{N_n} a_{n-k}^{(j)}h_k
=\sum_{k=0}^{N_n-1} a_{n-k}^{(j+1)}v_k+a_{n-N_n}^{(j)}v_{N_n}.
\end{equation}
Since $H\in\cH_m^r$ with $m\ge0$, $\sum_{\ell=0}^\infty h_\ell=0$ so we get
\begin{align}
\label{eq:partial.2}
\nonumber
|v_k|&=\left|\sum_{\ell=k+1}^\infty h_\ell\right|
\le |H(s)|_{\cH_m^r}\sum_{\ell=k+1}^\infty\frac{\log^r \ell}{\ell^{(m+3)/2}}\\
&\le |H(s)|_{\cH_m^r} \theta(r,m)\frac{\log^r k}{k^{(m+1)/2}},\quad k\ge2,
\end{align}
where 
$$
\theta(r,m):=\max_{k\ge2}\frac{k^{(m+1)/2}}{\log^r k}\sum_{\ell=k+1}^\infty\frac{\log^r \ell}{\ell^{(m+3)/2}}<\infty.
$$
Combining this with the estimate $|a_n^{(j)}|\le C_jn^{-j+1/2}$, we infer from 
\eqref{eq:partial.1} that 
$$
\left|\sum_{k=0}^{N_n} a_{n-k}^{(j)}h_k-\sum_{k=0}^{N_n} a_{n-k}^{(j+1)}v_k\right|
=|a_{n-N_n}^{(j)}v_{N_n}|
\le \tilde{\theta}(j,r,m)|H(s)|_{\cH_m^r}\frac{(\log n)^r}{n^{(m+2j)/2}}
$$
for some finite constant $\tilde{\theta}(j,r,m)$, which does not depend on the the sequence
$\{h_n\}$.

Now we can conclude that if the lemma is valid for the sequence $\{v_n\}$ then it is also valid for the sequence $\{h_n\}$. Noting that the generating function of $\{v_n\}$ belongs to $\cH_{m-2}^r$ for any $H\in\cH_{m}^r$ and any $m\ge2$, we see that it remains to prove the lemma for $m=0$ and $m=1$. Using \eqref{eq:partial.2}, one gets easily, for $m=0$ and $m=1$, 
\begin{align*}
\left|\sum_{k=0}^{N_n} a_{n-k}^{(j+1)}v_k\right|
&\le \sum_{k=0}^{N_n}|a_{n-k}^{(j+1)}||v_k|\\
&\le C_jn^{-j-1/2}(2|H(s)|_{\cH_m^r}+\sum_{k=2}^{N_n}|v_k|)
\\ &\le \Theta |H(s)|_{\cH_m^r} n^{-j-1/2}\frac{(\log n)^{r+\delta_m}}{n^{(m-1)/2}},
\end{align*}
where $\Theta$ is a constant depending only on $j$ and $r$. Thus, the proof is finished.\end{proof}
\begin{definition}
We shall say that a generating function $H$ belongs to the class
$\mathcal{R}_{k,m}^r,$ with $m\ge k\ge-1$ if there exist numbers
$q_i, i = k, \dots, m$ such that
\begin{equation} \label{hdecdef}
H(s)-\sum_{i=k}^{m} q_i (1-s)^{i/2} \in  \cH_m^r.
\end{equation}
\end{definition}
The following simple lemma underlines the importance of the class $\mathcal{R}_{k,m}^r$
for our approach.
\begin{lemma} \label{lem:rdecomp}
A generating function $H(s)$ belongs to the class $\mathcal{R}_{k,m}^r$ with $k=-1$ or $k=0$ if and only if
\begin{equation}
\label{eq:h_n-expan}
[s^n]H(s)=\sum_{j=\lfloor \frac{k+1}{2}\rfloor}^{\lfloor \frac{m+1}{2}\rfloor}\mu_{j}a_n^{(j+1)}
+O\left(\frac{\log^r n}{n^{(m+3)/2}}\right).
\end{equation}
\end{lemma}
\begin{proof}
Assume that $H(s)\in\mathcal{R}_{k,m}^r$. Set 
$$
\widetilde{H}(s)=H(s)-\sum_{i=k}^{m} q_i (1-s)^{i/2}.
$$
Then $\widetilde{H}(s)\in \cH_m^r$ and 
$$
H(s)=\sum_{i=k}^{m} q_i (1-s)^{i/2}+\widetilde{H}(s).
$$
Rewriting this equality in terms of coefficients we obtain \eqref{eq:h_n-expan}
with $\mu_j=q_{2j-1}$.

Assume now that \eqref{eq:h_n-expan} holds and denote
$$
\delta_n:=[s^n]H(s)-\sum_{j=\lfloor \frac{k+1}{2}\rfloor}^{j=\lfloor \frac{m+1}{2}\rfloor}\mu_{j}a_n^{(j+1)}.
$$
Therefore, one has the following equality for the generating functions:
$$
H(s)=\sum_{j=\lfloor \frac{k+1}{2}\rfloor}^{j=\lfloor \frac{m+1}{2}\rfloor}\mu_{j}(1-s)^{j-1/2}
+\sum_{n=0}^\infty \delta_n s^n.
$$
We know that $\delta_n=O\left(\frac{\log^r n}{n^{(m+3)/2}}\right)$, but the generating function of $\{\delta_n\}$ is not necessarily orthogonal to polynomials. But the function
$$
\sum_{n=0}^\infty \delta_n s^n
-\sum_{0\le\ell\le \frac{m}{2}}(1-s)^\ell(-1)^\ell\frac{d^\ell}{ds^\ell}\left(\sum_{n=0}^\infty \delta_n s^n\right)\Big|_{s=1}
$$
is orthogonal to any polynomial of degree which does not exceed $m/2$. This gives \eqref{hdecdef} with 
$$
q_{2i-1}=\mu_i\quad\text{and}\quad 
q_{2i}=(-1)^i\frac{d^i}{ds^i}\left(\sum_{n=0}^\infty \delta_n s^n\right)\Bigg|_{s=1}
= (-1)^i\sum_{n=i}^\infty \frac{n!}{(n-i)!}\delta_n.
$$
Thus, the proof is complete.
\end{proof}
\begin{lemma}
\label{lem:derivative}
For all  $m\ge k\ge 1$ one has
\begin{equation*}
H(s)-H(1) \in \cR_{k,m}^r \Longleftrightarrow \dH \in \cR_{k-2,m-2}^r.
\end{equation*}
\end{lemma}
\begin{proof}
The desired equivalence follows from the equality
\begin{equation*}
\frac{d}{ds}\left(\sum_{i=k}^{m} q_i (1-s)^{i/2}\right)
=\sum_{i=k}^{m} - \frac{i}{2} q_i(1-s)^{i/2-1}
=\sum_{i=k-2}^{m-2} \left(-\frac{i+2}{2} q_{i+2}\right)(1-s)^{i/2}
\end{equation*}
and from Lemma \ref{dlemma}.
\end{proof}
\begin{lemma} \label{rlemma}
If $H \in \cR_{k_0,m_0}^r$ and  $G \in \cR_{k_1,m_1}^r$ with $k_0 \le k_1$ and $m_0 \le m_1$, then
\begin{enumerate}
    \item $H + G \in \cR_{k_0,m_0}^r$,
    \smallskip
    \item $H \cdot G \in \cR_{k_0,m_0}^r\;$ if $\;k_0\ge0$,
    \smallskip
    \item $H \cdot G \in \cR_{k_1-1,m_0\wedge(m_1-1)}^{r+\textcolor{red}{\delta}} \;$ if $\;k_0 = -1$ and $ k_1 \ge 0$.
\end{enumerate}
where  $\delta = \mathds{1}
_{\{m_1\equiv 1 \mod 2\}} \cdot \mathds{1}_{m_1 \le m_0 +1}$.
\end{lemma}
\begin{proof}
The first statement is immediate from the relation
$\cR_{k_0,m_0}^r \subset \cR_{k_1,m_1}^r$.

It follows from the assumptions $H \in \cR_{k_0,m_0}^r$ and  $G \in \cR_{k_1,m_1}^r$ that 
$$
H(s)=\sum_{i=k_0}^{m_0} q^H_i (1-s)^{i/2} + H_0(s)
$$
and 
$$
G(s)=\sum_{i=k_1}^{m_1} q^G_i (1-s)^{i/2} + G_0(s),
$$
where
\begin{equation*}
H_0 \in \cH_{m_0}^r\quad\text{and}\quad G_0 \in \cH_{m_1}^r.
\end{equation*}
Therefore,
\begin{equation*}
H(s)G(s)
=\left(\sum_{i=k_0}^{m_0} q^H_i (1-s)^{i/2} + H_0(s)\right)
\left(\sum_{i=k_1}^{m_1} q^G_i (1-s)^{i/2} + G_0(s)\right).
\end{equation*}

If $k_0 \ge 0$ then the statement in (ii) follows from Lemmata~\ref{mlemma} and \ref{lem:hlemma_2}.

To prove the last statement we notice that 
\begin{align*}
H(s)G(s)
&=\sum_{i=-1}^{m_0}\sum_{j=k_1}^{m_1}q_i^Hq_j^G(1-s)^{(i+j)/2}
+\sum_{j=k_1}^{m_1}q_j^G(1-s)^{j/2}H_0(s)\\
&\hspace{1cm}+q^H_{-1}\frac{G_0(s)}{\sqrt{1-s}}
+\sum_{i=0}^{m_0}q_i^H(1-s)^{i/2}G_0(s)
+H_0(s)G_0(s).
\end{align*}
By Lemma \ref{lem:hlemma_3}, 
\begin{equation*}
\frac{G_0(s)}{\sqrt{1-s}} \in \cH_{m_1-1}^{r+\delta_{m_1}}.
\end{equation*}
Moreover, by Lemma~\ref{lem:hlemma_2},
$$
\sum_{i=0}^{m_0-1}q_i^H(1-s)^{i/2}G_0(s)\in \cH_{m_1}^r
$$
and 
$$
\sum_{j=k_1}^{m_1-1}q_j^G(1-s)^{j/2}H_0(s)\in\cH_{m_0}^r.
$$
Finally, by Lemma \ref{mlemma}, $H_0(s)G_0(s)$ belongs to $\cH_{m_0}^r$. So we can write
\begin{align*}
H(s)G(s)-\sum_{i=-1}^{m_0}\sum_{j=k_1}^{m_1}q_i^Hq_j^G(1-s)^{(i+j)/2} \in \cH_{m_0}^r \cup \cH_{m_1-1}^{r+\delta_{m_1}}
\end{align*}
If $m_1 > m_0 + 1$ then $\cH_{m_1-1}^{r+\delta_{m_1}} \subset \cH_{m_0}^r$. Otherwise 
   $ \cH_{m_0}^r \subset \cH_{m_1-1}^{r+\delta_{m_1}}$.
So it is the only rest to mention that the minimal value of $(i+j)$ equals $k_1-1$. 
\end{proof}
\begin{lemma}
\label{lem:exp1-s}
If $k\ge1$ and $\lambda\in\R$ then 
\begin{equation*}
e^{\lambda(1-s)^{k/2}} \in \cR_{0,m}^0\quad\text{for every } m\ge1.
\end{equation*}
\end{lemma}
\begin{proof}
By definition of the exponent:
\begin{align*}
e^{\lambda(1-s)^{k/2}} 
&=\sum_{n=0}^m \frac{\lambda^n (1-s)^{kn/2} }{n!}
+\sum_{n=m+1}^\infty\frac{\lambda^n (1-s)^{kn/2} }{n!}\\
&=\sum_{n=0}^m \frac{\lambda^n (1-s)^{kn/2} }{n!}
+\sum_{r=0}^{m}\sum_{l=1}^\infty\frac{\lambda^{l(m+1)+r}(1-s)^{k(l(m+1)+r)/2}}{(l(m+1)+r)!}.
\end{align*}
For every fixed $r$,
\begin{equation*}
\sum_{l=1}^\infty\frac{\lambda^{l(m+1)+r}(1-s)^{k(l(m+1)+r)/2} }{(l(m+1)+r)!} 
=\lambda^r (1-s)^{k r/2}\sum_{l=1}^\infty
\frac{\left(\lambda^{(m+1)} (1-s)^{k(m+1)/2} \right)^l}{((m+1)l + r)!}.
\end{equation*}
By Lemma~\ref{1-slemma}, $(1-s)^{k (m+1)/2}  \in \cH^0_{k(m+1)-1}\subset \cH_{m}^0$.
Applying  now Lemma~\ref{explemma}, we have
\begin{equation*}
\sum_{l=1}^\infty
\frac{\left(\lambda^{(m+1)} (1-s)^{k(m+1)/2} \right)^l}{((m+1)l+r)!}
\in \cH^0_{m} .
\end{equation*}
Consequently, since by lemma \ref{lem:hlemma_2} multiplication by $(1-s)^{kr/2}$ does not affect parameter $m$:
\begin{align*}
\lambda^r (1-s)^{k r/2}\sum_{l=1}^\infty
\frac{\left(\lambda^{(m+1)} (1-s)^{k(m+1)/2} \right)^l}{((m+1)l + r)!}\in \cH^0_{m}
\end{align*}
Then by summation over $r$ we obtain the desired
\begin{equation*}
e^{\lambda(1-s)^{k/2}}-\sum_{n=0}^{m} \frac{\lambda^n (1-s)^{kn/2} }{n!}
\in\cH_{m}^0.
\end{equation*}
\end{proof}
%%%%%%%%%%%%%%%%%%%%%%%%%%%%%%%%%%%%%%%%%%%%%%%%%%%%%%%%%%%%%%%%%%%%%%%%%%%%%%%%%%%%%%%%%%%%%%%%%%%%%%%%%%%%%%%%%%%%%%%%%%%%%%%%%%%%%%%%%%%%%%%%%%%%%%
\section{Asymptotic expansions for $\tau_0$}
\label{sec:tau0}
\subsection{Proof of Theorem~\ref{thm:tau_0}}
As we have already mentioned in the introduction, we are going to use the representation \eqref{eq:WF2}.  To this end we first derive an appropriate for our purposes expression for the generating function of the sequence $\frac{\Delta_n}{n}$. According to Proposition~\ref{prop:S_n<0} with $r+3$ finite moments,
$$
\frac{\Delta_n}{n}=\sum_{j=1}^{\lfloor\frac{r}{2}\rfloor+1}
\theta_ja^{(j+1)}_{{n-1}}+h_n
\quad\text{and}\quad 
h_n=o(n^{-(r+3)/2}).
$$
Hence  
\begin{align}
\label{eq:Delta-repr}
\nonumber
\sum_{n=1}^\infty \frac{\Delta_n}{n}s^n 
&=\sum_{j=1}^{\lfloor\frac{r}{2}\rfloor+1}\theta_j
\sum_{n=1}^\infty a^{(j+1)}_{{n-1}}s^n+\sum_{n=1}^\infty h_n s^n\\
&=\sum_{j=1}^{\lfloor\frac{r}{2}\rfloor+1}\theta_j(1-s)^{j-1/2}
+\sum_{n=1}^\infty h_n s^n.
\end{align}
Set $H(s)=\sum_{n=1}^\infty h_n s^n$. The property $h_n=o(n^{-(r+3)/2})$ implies that $H^{(k)}(1)$ is finite for every $k\le \lf r/2\rf$. We now set
$\psi_k:=(-1)^k\frac{H^{(k)}(1)}{k!}$
and
$$
H_0(s):=H(s)-\sum_{k=0}^{\lf r/2\rf}\psi_k(1-s)^k.
$$
Then
$$
H_0^{(k)}(1)=0
\quad\text{for all }k=0,1,\ldots,\lf r/2\rf.
$$
In other words, $H_0(s)$ is 'orthogonal' to all polynomials of degree 
$\lf r/2\rf$. Since the coefficients of this generating function are of order $o(n^{-(r+3)/2})$, we conclude that $H_0(s)\in \cH_{r}^0$.
Combining this with \eqref{eq:Delta-repr}, we conclude that 
$$
\sum_{n=1}^\infty \frac{\Delta_n}{n}s^n=
\sum_{j=1}^{\lfloor\frac{r}{2}\rfloor+1}\theta_j(1-s)^{j-1/2}
+\sum_{k=0}^{\lfloor\frac{r}{2}\rfloor}\psi_k(1-s)^k+H_0(s).
$$
Equivalently, $\sum_{n=1}^\infty \frac{\Delta_n}{n}s^n$ belongs to
$\cR_{0,r}^0$. Consider the corresponding decomposition:
$$
\sum_{n=1}^\infty \frac{\Delta_n}{n}s^n
=\sum_{i=0}^{r}q_i(1-s)^{i/2}+F(s)=:Q(s)+F(s).
$$
Plugging this representation into \eqref{eq:WF2}, we get 
\begin{align}
\label{eq:WF3}
\sum_{n=0}^\infty\pr(\tau_0>n)s^n
=\frac{e^{Q(s)+F(s)}}{\sqrt{1-s}}
=\frac{e^{Q(s)}}{\sqrt{1-s}}+\frac{e^{F(s)}-1}{\sqrt{1-s}}e^{Q(s)}.
\end{align}
Combining Lemma~\ref{lem:exp1-s} and Lemma~\ref{rlemma}(ii), we conclude that $e^{Q(s)}$ belongs to $\cR_{0,m}^0$ for every $m\ge 1$. Choosing $m=r$  we get the representation
\begin{align}
    e^{Q(s)}=\sum_{i=0}^{r}\mu_i(1-s)^{i/2}+G(s),
\quad G(s)\in\cH_{{r}}^0,
\end{align}
where $\mu_i, i=1,\cdots,r$ are some real numbers.
By Lemma \ref{lem:hlemma_3},
\begin{align*}
    \frac{e^{Q(s)}}{\sqrt{1-s}}=\sum_{i=0}^{r}\mu_i(1-s)^{(i-1)/2}+\hat{G}(s),
\quad \hat{G}(s)\in\cH_{{r-1}}^{\delta_r}.
\end{align*}
Recalling that $F(s)\in\cH_{r}^0$ and using Lemma~\ref{explemma}, we infer that $e^{F(s)}-1\in\cH_{r}^0$ and, by Lemma~\ref{lem:hlemma_3}, 
$$
\frac{e^{F(s)}-1}{\sqrt{1-s}}\in\cH_{r-1}^{{\delta_r}}.
$$
Applying Lemma~\ref{mlemma} and Lemma~\ref{lem:hlemma_2}, we obtain 
$$
\frac{e^{F(s)}-1}{\sqrt{1-s}}e^{Q(s)}\in\cH_{r-1}^{{\delta_{r}}}.
$$
Applying these properties to the corresponding terms in \eqref{eq:WF3}, we have 
\begin{align}
\label{eq:g.f.tau0}
\sum_{n=0}^\infty\pr(\tau_0>n)s^n-\sum_{i=0}^{r}\mu_i(1-s)^{(i-1)/2}
\in\cH_{r-1}^{\delta_r}.
\end{align}
In other words, $\sum_{n=0}^\infty\pr(\tau_0>n)s^n$ belongs to $\cR_{-1,r-1}^{\delta_r}$.
Applying now Lemma~\ref{lem:rdecomp} with $k=-1$ and $m=r-1$, we obtain
$$
\pr(\tau_0>n)=\sum_{j=0}^{\lf r/2\rf}\mu_{2j-1}a_n^{(j+1)}
+
O\left(\frac{\log^{{\delta_r}} n}{n^{(r+2)/2}}\right).
$$
Changing the index of summation we get the desired expansion with coeficients $\nu_j=\mu_{2j-3}$.
%%%%%%%%%%%%%%%%%%%%%%%%%%%%%%%%%%%%%%%%%%%%%%%%%%%%%%%%%%%%%%%%%%%%%%%%%%%%%%%%%%%%%%%%%%%%%%%%%%%%%%%%%%%%%%%%%%%%%%%%%%%%%%%%%%%%%%%%%%%%%%%%%%%%%%
%%%%%%%%%%%%%%%%%%%%%%%%%%%%%%%%%%%%%%%%%%%%%%%%%%%%%%%%%%%%%%%%%%%%%%%%%%%%%%%%%%%%%%%%%%%%%%%%%%%%%%%%%%%%%%%%%%%%%%%%%%%%%%%%%%%%%%%%%%%%%%%%%%%%%%

\subsection{Calculation of the first two correction coefficients}
In this paragraph we show how one can compute the coefficients in the expansion for $\tau_0$ and derive exact expressions for first two correction terms $\nu_2$ and $\nu_3$. Since the polynomials $Q_1$, $Q_3$,
$\widetilde{Q}_1$ and $\widetilde{Q}_3$ are coming from asymptotic expansions in the central limit theorem, we will consider these functions as a given input. We first compute the numbers $\theta_1$ and $\theta_2$ in Proposition~\ref{prop:S_n<0}.
According to Theorem VI.1 in Flajolet and Sedgewick \cite{FlajoletSedgewick},
$$
a_n^{(2)}
=\frac{n^{-3/2}}{\Gamma(-1/2)}\left(1+\sum_{k=1}^\infty\frac{e_k(2)}{n^k}\right)
=\frac{n^{-3/2}}{\Gamma(-1/2)}+\frac{3}{8}\frac{n^{-5/2}}{\Gamma(-1/2)}+O(n^{-7/2})
$$
and
$$
a_n^{(3)}
=\frac{n^{-5/2}}{\Gamma(-3/2)}\left(1+\sum_{k=1}^\infty\frac{e_k(3)}{n^k}\right)
=\frac{n^{-5/2}}{\Gamma(-3/2)}+O(n^{-7/2}).
$$
Therefore,
$$
n^{-5/2}=\Gamma(-3/2)a_n^{(3)}+O(n^{-7/2})
$$
and
\begin{align*}
n^{-3/2}&=\Gamma(-1/2)a_n^{(2)}-\frac{3}{8}n^{-5/2}+O(n^{-7/2})\\
&=\Gamma(-1/2)a_n^{(2)}-\frac{3}{8}\Gamma(-3/2)a_n^{(3)}+O(n^{-7/2}).
\end{align*}
Plugging these expressions into \eqref{eq:asymp2}, we get 
\begin{align*}
\frac{\Delta_n}{n}
&=-\frac{Q_1(0)}{n^{3/2}}-\frac{Q_3(0)}{n^{5/2}}+o(n^{-5/2})\\
&=-Q_1(0)\Gamma(-1/2)a_n^{(2)}
-\Bigl(Q_3(0)-\frac{3}{8}Q_1(0)\Bigr)\Gamma(-3/2)a_n^{(3)}+o(n^{-5/2}).
\end{align*}
This implies that 
$$
\theta_1 = 2 \sqrt{\pi} Q_1(0), \quad \theta_2 = -\frac{4\sqrt{\pi}}{3}\big(Q_3(0) - \frac{3}{8} Q_1(0)\big).
$$

Next determine the numbers $\psi_k$ for $k=0,1,2$. By the definition,
$$
\psi_k=(-1)^kH^{(k)}(1),
$$
where $H$ is the function defined directly after
\eqref{eq:Delta-repr}. 
It is easy to see that 
\begin{align*}
\psi_0 = \sum_{n=1}^\infty \frac{\Delta_n}{n},
\end{align*}
\begin{align*}
\psi_1 = - \sum_{n=1}^\infty n \cdot \left( \frac{\Delta_n}{n} - 2 \sqrt{\pi} Q_1(0) a_n^{(2)}\right)
\end{align*}
and
\begin{align*}
    \psi_2 = \sum_{n=1}^\infty \frac{n(n-1)}{2}\cdot \left( \frac{\Delta_n}{n} 
    {-}2 \sqrt{\pi} Q_1(0) a_n^{(2)}+\frac{4\sqrt{\pi}}{3}\left(Q_3(0) - \frac{3}{8}Q_1(0)\right) a_n^{(3)}\right).
\end{align*}
In the case $r=4$, the function $Q(s)$ from the proof of Theorem~\ref{thm:tau_0} is given by
\begin{align*}
    %\sum_{n=1}^\infty 
    %\frac{\Delta_n}{n}s^n 
    Q(s) = \sum_{i=0}^2 
    \psi_i (1-s)^i + \sum_{i=1}^2 
    \theta_i (1-s)^{i-1/2}.
\end{align*}
Then
\begin{align}
\label{eq:eQ}
e^{Q(s)} = e^{\psi_0}\prod_{i=1}^2 e^{ \phi_i (1-s)^i}\prod_{i=1}^2 e^{\theta_i (1-s)^{i-1/2}}.
\end{align}
Next we notice that
$$
e^{\psi_1(1-s)} = 1 + \psi_1(1-s) + \frac{(\psi_1)^2}{2}(1-s)^2+ %(\psi_1)^3(1-s)^3 + 
\cH_4^0,
$$
$$
e^{\psi_2(1-s)^2} = 1 +  \psi_2(1-s)^2 +\cH_4^0,
$$
$$
e^{\theta_1(1-s)^{1/2}} = 1 + 
\sum_{i=1}^4 (\theta_1)^i (1-s)^{i/2}
+\cH_4^0
$$
and
$$
e^{\theta_2(1-s)^{3/2}} = 1 + \theta_2(1-s)^{3/2} + 
\cH_4^0,
$$
Applying these representations to the corresponding parts in \eqref{eq:eQ}
and using Lemma~\ref{lem:hlemma_2}, we conclude that
\begin{align*}
    e^{Q(s)} = e^{\psi_0}\sum_{i=0}^4 \mu_{i}(1-s)^{i} + \cH_4^0,
\end{align*}
where
$$
\mu_{0} = 1, \quad \mu_1 =\theta_1, 
$$ 
$$
\mu_2 = \psi_1 + (\theta_1)^2,
$$
$$ 
\mu_3 =  \theta_2 + \psi_1 \theta_1 + (\theta_1)^3
$$
and
$$ 
\mu_4 = \psi_2 + \frac{(\psi_1)^2}{2} + \psi_1 (\theta_1)^2 + \theta_2 \theta_1 + (\theta_1)^4.
$$

% $$ r_5 = \theta_3 + \psi_2\theta_1 + \frac{(\psi_1)^2}{2}\theta_1 + \theta_2\psi_1 + \theta_2(\theta_1)^2 + 
% \psi_1 (\theta_1)^3 + (\theta_1)^5.$$
Finally it means that
$$
\sum_{n=0}^\infty\pr(\tau_0>n)s^n = \sum_{i=0}^{4}e^{\psi_0}\mu_i(1-s)^{(i-1)/2}
+ \cH_3^0.
$$
For the coefficients we have then the following representation:
\begin{align*}
\pr(\tau_0>n)s^n = &e^{\psi_0}a_n^{(1)} + e^{\psi_0}(\psi_1 + (\theta_1)^2) a_n^{(2)}\\ 
    &+
    e^{\psi_0}\left(\psi_2 + \frac{(\psi_1)^2}{2} + \psi_1 (\theta_1)^2 + \theta_2 \theta_1 + (\theta_1)^4\right) a_n^{(3)} 
    +O\left({n^{-3}} \right).
\end{align*}
%%%%%%%%%%%%%%%%%%%%%%%%%%%%%%%%%%%%%%%%%%%%%%%%%%%%%%%%%%%%%%%%%%%%%%%%%%%%%%%%%%%%%%%%%%%%%%%%%%%%%%%%%%%%%%%%%%%%%%%%%%%%%%%%%%%%%%%%%%%%%%%%%%%%%%
\section{Expansions for conditioned probabilities}
\label{sec:loc_prob}
\subsection{Proof of Theorem~\ref{thm:cond_prob} in the case of fixed $x$}
\label{subs:expansion}
We shall prove \eqref{eq:cp1} only, since \eqref{eq:cp2} can be proven by absolutely the
same arguments. Moreover, we restrict our attention to the lattice case: we shall always assume that $X_1$ takes values on $\mathbb{Z}$ and that $\mathbb{Z}$ is the minimal lattice for $X_1$.

According to Lemma 17 from \cite{VW09}, the function $b_n(x):=\pr(S_n=x,\tau_0>n)$ satisfies the recurrence equation
\begin{align} \label{eq:brecurrent}
   n b_n(x)=\pr(S_n=x)+\sum_{k=1}^{n-1}\sum_{y=1}^{x-1}\pr(S_k=y)b_{n-k}(x-y).
\end{align}
For density case one can obtain analog of this relation from Lemma 15 \cite{VW09}. Multiplying both sides of equation \eqref{eq:brecurrent} with $s^{n-1}$ and taking the sum over all $n\ge1$,
we get 
\begin{align*}
\sum_{n=1}^\infty ns^{n-1}b_n(x)
&=\sum_{n=1}^\infty s^{n-1}\pr(S_n=x)
+\sum_{n=1}^\infty\sum_{k=1}^{n-1}\sum_{y=1}^{x-1}
s^{n-1}\pr(S_k=y)b_{n-k}(x-y)\\
&=\sum_{n=1}^\infty s^{n-1}\pr(S_n=x)
+\sum_{y=1}^{x-1}\sum_{k=1}^\infty s^{k-1}\pr(S_k=y)
\sum_{n=k+1}^\infty s^{n-k}b_{n-k}(y)\\
&=\sum_{n=1}^\infty s^{n-1}\pr(S_n=x)
+\sum_{y=1}^{x-1}\sum_{k=1}^\infty s^{k-1}\pr(S_k=y)
\sum_{n=1}^\infty s^{n}b_{n}(y).
\end{align*}
Setting now 
$$
B(y,s)=\sum_{n=1}^\infty s^{n}b_{n}(y)
$$
and
$$
\Psi(y,s)=\sum_{n=1}^\infty s^{n-1}\pr(S_n=y),
$$
we obtain the equation
\begin{align}
\label{eq:B-equation}
\frac{d}{d s}B(x,s)=\Psi(x,s)+\sum_{y=1}^{x-1}\Psi(y,s)B(x-y,s).
\end{align}

By Lemma 20 from \cite{VW09}, $b_n(x)\le C(x)n^{-3/2}$. Moreover, 
$\sum\limits_{n=1}^\infty b_n(x)$ is finite and equal to the mass function of the renewal function associated with strict increasing ladder epochs. Combining these facts, we infer that $B(x,s)-B(x,1)$ belongs to $\cH_0^0$. Consequently, $B(x,s)\in\cR_{0,0}^0$.

By Proposition \ref{prop:asymplattice},
$$
\pr(S_n=x)=\sum_{j=0}^{\lfloor r/2\rfloor}\theta_j(x)a_{n-1}^{(j+1)}+h_n(x),
$$
where
\begin{align*}
    |h_n(x)| \le C_r \frac{(|x|+1)^{r+1}}{n^{(r+2)/2}}.
\end{align*}
This representation implies that
\begin{align*}
\Psi(x,s)&=\sum_{j=0}^{\lf r/2\rf}\theta_j(x)
\sum_{n=1}^\infty s^{n-1}a_{n-1}^{(j+1)}+\sum_{n=1}^\infty s^nh_n(x)\\
&=\sum_{j=0}^{\lf r/2\rf}\theta_j(x)(1-s)^{j-1/2}
+
\sum_{n=1}^\infty s^nh_n(x).
\end{align*}
Next we define $\psi_j(x)$ as in Lemma \ref{lem:rdecomp}, namely
$\psi_{2j-1}(x)=\theta_j(x)$ and
\begin{align*}
    \psi_{2j}(x) = (-1)^i\sum_{n=i}^\infty \frac{n!}{(n-i)!}h_n.
\end{align*}
Then one has
\begin{align}
\label{eq:psi-decomp}
\Psi(x,s)=\sum_{j=-1}^{r-1}\psi_j(x)(1-s)^{j/2}+H_\Psi(x,s), \quad H_\Psi\in\cH_{r-1}^0.
\end{align}
Therefore, $\Psi(x,s)\in\cR_{-1,r-1}^0$. By Lemma~\ref{rlemma}, since $B(x,s) \in \cR_{0,0}^0$, the generating function
$\Psi(x,s)+\sum_{y=1}^{x-1}\Psi(y,s)B(x-y,s)$
belongs to $\cR_{-1,-1}^1$. Taking into account \eqref{eq:B-equation}
and applying Lemma~\ref{lem:derivative}, we conclude that 
$B(x,s)-B(x,1)\in \cR_{1,1}^1$ and, consequently, $B(x,s)\in\cR_{0,1}^1$. 
Applying Lemma~\ref{rlemma} once again, we infer that 
$\Psi(x,s)+\sum_{y=1}^{x-1}\Psi(y,s)B(x-y,s)$ belongs to $\cR_{-1,0}^2$. Lemma~\ref{lem:derivative} gives now $B(x,s)-B(s,1)\in\cR_{1,2}^2$ and 
$B(x,s)\in\cR_{0,2}^2$. Iterating this, we get 
$B(x,s)\in\cR_{0,r}^{r}$. This implies that $B(x,s)\in\cR_{0,r-1}^0$. Applying the described
above procedure one more time, we finally arrive at 
$$
B(x,s)\in\cR_{0,r}^{\delta_{r}}.
$$ 
This implies that there exist functions $q_i(x)$ such that
\begin{align}
\label{eq:B-decomp}
B(x,s)=\sum_{i=0}^{r} q_i(x)(1-s)^{i/2}+H_B(x,s), \quad  H_B\in\cH_{r}^{\delta_{r}}.
\end{align}
Applying now Lemma~\ref{lem:rdecomp}, we conclude that, for every fixed $x$,
\begin{align*}
    \pr(S_n=x, \tau_0 > n) = \sum_{j=1 }^{\lf \frac{r+1}{2} \rf}
    U_j(x)a_n^{(j+1)} + O\left(\frac{\log^{\delta_{r}} n}{n^{(r+3)/2}}\right).
\end{align*}
where $U_j(x) := q_{2j-1}(x)$.

To complete the proof of the theorem for fixed values of $x$ it remains to derive recurrence relation for the coefficients $q_j$. The case when $x$ is not necessarily bounded is more complicated and will be considered in the subsequent subsection.

Plugging \eqref{eq:B-decomp} and \eqref{eq:psi-decomp} into \eqref{eq:B-equation}, we have
\begin{align}
\label{eq:q-equation}
\nonumber
&-\sum_{i=1}^{r} q_i(x)\frac{i}{2}(1-s)^{i/2-1}+\frac{d}{ds}H_B(x,s)\\
\nonumber
&=\sum_{j=-1}^{r}\psi_j(x)(1-s)^{j/2}
+\sum_{y=1}^{x-1}\left(\sum_{j=-1}^{r}\psi_j(y)(1-s)^{j/2}\right)
\left(\sum_{i=0}^{r} q_i(x-y)(1-s)^{i/2}\right)\\
&\hspace{1cm}+H_\Psi(x)
+\sum_{y=1}^{x-1}\left[H_\Psi(y,s)B(x-y,s)+\Psi(y,s)H_B(x-y,s)\right].
\end{align}
Recalling that $B(y,s)\in\cR_{0,{r}}^1$, $\Psi(y,s)\in\cR_{-1,{r-1}}^0$ and applying Lemma~\ref{rlemma}(iii), we conclude that 
$$
H(x,s):=H_\Psi(x)
+\sum_{y=1}^{x-1}\left[H_\Psi(y,s)B(x-y,s)+\Psi(y,s)H_B(x-y,s)\right]
\in\cH_{r-1}^1.
$$ 
Multiplying \eqref{eq:q-equation} by $(1-s)^{1/2}$ and letting $s\to1$ we get the equality
\begin{equation}
\label{eq:q1}
-\frac{1}{2}q_1(x)=\psi_{-1}(x)+\sum_{y=1}^{x-1}\psi_{-1}(y)q_0(x-y).
\end{equation}
We know that $q_0(y)=\sum_{n=1}^\infty\pr(S_n=y,\tau_0>n)$ for every $y\ge1$.
Moreover, $\psi_{-1}(x)\equiv\frac{1}{\sqrt{2\pi}}$. Therefore, 
\begin{align*}
q_1(x)&=-2\frac{1}{\sqrt{2\pi}}\left(1+\sum_{y=1}^{x-1}q_0(x)\right)
=-\sqrt{\frac{2}{\pi}}\left(1+\sum_{n=1}^\infty\pr(S_n<x,\tau_0>n)\right)\\
&=-\sqrt{\frac{2}{\pi}}V(x).
\end{align*}
Combining \eqref{eq:q-equation} and \eqref{eq:q1},
we get 
\begin{align*}
&-\sum_{i=2}^{r} q_i(x)\frac{i}{2}(1-s)^{i/2-1}+\frac{d}{ds}H_B(x,s)\\
&=\sum_{j=0}^{r}\psi_j(x)(1-s)^{j/2}
+\sum_{y=1}^{x-1}\sum_{j=-1}^{r}\sum_{i=0}^{r}\psi_j(y)q_i(x)(1-s)^{(i+j)/2}{\rm 1}\{i+j\ge0\}\\
&\hspace{1cm}+H_\Psi(x)
+\sum_{y=1}^{x-1}\left[H_\Psi(y,s)B(x-y,s)+\Psi(y,s)H_B(x-y,s)\right].
\end{align*}
Letting $s\to1$, we conclude
that 
$$
-q_2(x)=\psi_0(x)+\sum_{y=1}^{x-1}\left(\psi_{-1}(y)q_1(x-y)
+\psi_0(y)q_0(x-y)\right).
$$
Substituting that once again in \eqref{eq:q-equation} and dividing by $(1-s)^{1/2}$ we obtain relation for $q_3(x)$.
Repeating this argument, we get for every $\ell \le r$ the equality 
\begin{align} \label{eq:qrecurrent}
   -\frac{\ell}{2}q_\ell(x)=\psi_{\ell-2}(x)+
\sum_{y=1}^{x-1}\sum_{j=-1}^{\ell-2}\psi_j(y)q_{\ell-2-j}(x-y). 
\end{align}
Thus, we can compute all $q_\ell$ recursively.

\subsection{Proof of Theorem~\ref{thm:cond_prob} for growing $x$.} Hier we always assume that $x\ge 0$ and write $C$ for any constant which are independent of $n$ and $x$, but may depend on other parameters of the random walk.

We first estimate the coefficients $\psi_j(x)$ in the representation 
\eqref{eq:psi-decomp}.
According to Proposition \ref{prop:asymplattice}, $\psi_{2j-1}(x) = \theta_j(x)$ is a polynomial of degree $2j$. This means, in particular, that  
\begin{align}
\label{eq:theta1}
|\psi_{2j-1}(x)|\le C(1+x^{2j}), \quad \quad j = 0, \dots, \lf r/2 \rf.
\end{align}
\begin{lemma} \label{lem:psi2j}
For every $j\le(r-1)/2$ one has
\begin{align} \label{ineq:psi2j}
    |\psi_{2j}(x)| \le C(1+x^{2j+1}).
\end{align}
\end{lemma}
\begin{proof}

It follows from \eqref{eq:psi-decomp} that 
$$
\psi_{2j}(x)=(-1)^j\frac{d^{j}}{ds^{j}}\left(\sum_{n=1}^\infty h_{n+1}(x)s^n\right)\Big|_{s=1}.
$$
Consequently,
\begin{align}
\label{eq:theta2}
|\psi_{2j}(x)|\le \sum_{n=1}^\infty n^{j}|h_{n+1}(x)|.
\end{align}
From now on fix some $j$ and consider Proposition \ref{prop:asymplattice} with $r=2j+1$, then for corresponding sequence $h_n$ it holds
\begin{equation*}
    |h_{n+1}(x)| \le C \frac{(1 +x)^{2j+2}}{n^{(2j+3)/2}}
\end{equation*}
and, consequently,
\begin{align*}
    \sum_{n > x^2}^\infty h_{n+1}(x)n^{j} 
    \le C (1 + x)^{2j+2}\sum_{n >  x^2}^\infty 
    \frac{1}{n^{3/2}}. 
\end{align*}
It is clear that
\begin{equation*}
    \sum_{n > x^2}^\infty \frac{1}{n^{3/2}} 
    \le \int_{x^2}^\infty t^{-3/2} d t 
    =\frac{2}{x},\quad x\ge 1.
\end{equation*}
Therefore,
\begin{align}
\label{eq:theta3}
    \sum_{n > x^2}^\infty h_{n+1}(x)n^{j} \le C(1 + x)^{2j+1}.
\end{align}
For $n\le x^2$ we have the following obvious estimate
\begin{align*}
    |h_{n+1}(x)| \le |p_{n+1}(x)| + \sum_{\ell=0}^{j} | \psi_{2\ell-1}(x)a_{n}^{(\ell+1)}|.
\end{align*}
First of all, using the concentration bound for $p_{n+1}(x)$, we have
\begin{align*}
    |p_{n+1}(x)| \le \frac{C}{\sqrt{n}}
    \le C \frac{(1+x)^{2j}}{n^{(2j+1)/2}},\quad n\le x^2.
\end{align*}
Next, since $\psi_{2\ell-1}(x)$ is of degree $2\ell$, 
\begin{align*}
     |\psi_{2\ell - 1}(x) a_{n}^{(\ell +1 )}| \le C \frac{(1+x)^{2\ell}}{n^{\ell + 1/2}} \le C \frac{(1+x)^{2j}}{n^{(2j+1)/2}}.
\end{align*}
This yields
\begin{equation}
\label{eq:theta4}
    \sum_{n=1}^{x^2} |h_{n+1}(x)|n^{j} \le C (1+x)^{2j} \sum_{n=1}^{x^2} \frac{1}{\sqrt{n}} \le C (1+x)^{2j+1}.
\end{equation}
The proof is completed by plugging \eqref{eq:theta3} and \eqref{eq:theta4} into \eqref{eq:theta2}.
\end{proof}
Also we need an estimation for the functions $q_j(x)$:
\begin{lemma} \label{lem:qpower}
    For every $j\le r$ one has
\begin{align}\label{eq:qpower}
    |q_j(x)| \le C (1+x)^j.
\end{align}
\end{lemma}
\begin{proof}
    Recall that 
\begin{align*}
q_0(x)=\sum_{n\ge0}\pr(S_n=x,\tau_0>n)=V(x+1)-V(x)
\end{align*}
and 
$$
q_1(x)=-\sqrt{\frac{2}{\pi}}V(x),
$$
where $V(x)$ is the renewal function corresponding to ascending ladder heights
$\{\chi^+_k\}$.
By the local renewal theorem, $V(x+1)-V(x)$ converges, as $x\to\infty$, to 
$1/\e\chi^+_1$. This implies that
\begin{align*}
|q_0(x)|\le C\quad\text{and}\quad 
|q_1(x)|\le C(1+x).
\end{align*}
Furthermore, we know that $|\psi_j(x)| \le C(1 + x^{j+1})$ for all $j$.
Plugging these estimates into \eqref{eq:qrecurrent}, one obtains recursively
\begin{align*}
|q_j(x)| \le C(1+x^{j}) \quad j=2,3,\ldots,r.
\end{align*}
\end{proof}
Now we can finally start proving Theorem.
Let us rewrite \eqref{eq:brecurrent} in the following way:
\begin{align}
\label{eq:b_n-new}
   n b_n(x)=p_n(x)+\sum_{k=1}^{n-1}\sum_{y=1}^{x-1}p_k(y)b_{n-k}(x-y),
\end{align}
where $p_n(x) = \pr (S_n = x)$.
We split elements of both sequences into two parts, regular part and a remainder.
More precisely, we define
\begin{align*}
    b_n^{\text{reg}}(x) := \sum_{j=0}^r q_j(x)[s^n](1-s)^{j/2}, 
    \quad 
    b_n^{\text{rest}}(x) := b_n(x) -b_n^{\text{reg}}(x), \quad n\ge 0
\end{align*}
and 
\begin{align*}
    p_n^{\text{reg}}(x) = \sum_{j=-1}^{r-1} \psi_j(x) [s^{n-1}](1-s)^{j/2},
    \quad 
    p_n^{\text{rest}}(x) = p_n(x) -p_n^{\text{reg}}(x), \quad n\ge 1.
\end{align*}
Note that $b_0(x)=0$ for $x\ge1$ but it is not true anymore for $b_0^{\text{reg}}$ and $b_0^{\text{rest}}$.
Combining Proposition \ref{prop:asymplattice} and Lemma \ref{lem:psi2j} gives us the bound
\begin{align}
\label{eq:p_n^rest}
|p^{\text{rest}}_n(x)| \le C \frac{(1 + x)^{r+1}}{n^{(r+2)/2}},\quad n\ge1.
\end{align}
The sequence $b_n^{\text{reg}}$ almost satisfy the same relation as $b_n$ itself. Namely the following lemma holds.
\begin{lemma} \label{lem:bnregrec}
For $n \ge r+1$ it holds
\begin{align} \label{eq:bnregrec}
    \left|
    nb_n^{\text{reg}}(x) -p^{\text{reg}}_{n}(x) - 
\sum_{y=1}^{x-1}\sum_{k=1}^{n} p^{\text{reg}}_k(y) b^{\text{reg}}_{n-k}(x-y)
\right| \le C\frac{(1+x)^{r+1}}{n^{(r+1)/2}},\quad x\le\sqrt{n}.
\end{align}
\end{lemma}
\begin{proof}
Define
\begin{align*}
    B^{\text{reg}}(x,s) = \sum_{n=0}^\infty b_n^{\text{reg}}(x)s^n = \sum_{\ell=0}^{r}q_\ell(x)(1-s)^{\ell/2}
\end{align*}
and
\begin{align*}
    \Psi^{\text{reg}}(x,s) = \sum_{n=1}^\infty p_n^{\text{reg}}(x)s^{n-1} = \sum_{\ell=0}^{r-1}\psi_\ell(x)(1-s)^{\ell/2}.
\end{align*}
Then, using the recurrent relation \eqref{eq:qrecurrent}, one has, with $\omega=1-s$,
\begin{align*}
    \frac{d}{ds}B^{\text{reg}}(x,s) 
    =&
    \sum_{\ell=1}^{r}-\frac{\ell}{2}q_\ell(x)\omega^{\ell/2-1}=
    \sum_{\ell=-1}^{r-2}-\frac{\ell+2}{2}q_{\ell+2}(x)\omega^{\ell/2}\\
    =&\sum_{\ell=-1}^{r-2} 
    \left(
    \psi_\ell(x) + \sum_{y=1}^{x-1} \sum_{j=-1}^\ell \psi_j(y)q_{\ell-j}(x-y)
    \right)
    \omega^{\ell/2}\\
    =& \Psi^{\text{reg}}(x,s) - \psi_{r-1}(x)\omega^{(r-1)/2}\\
    +&\sum_{\ell=-1}^{r-2}\sum_{y=1}^{x-1}\sum_{j=-1}^\ell \psi_j(y)\omega^{j/2}\cdot q_{\ell-j}(x-y)\omega^{(\ell-j)/2}.
\end{align*}
For fixed $y$ one has
\begin{multline*}
    \sum_{\ell=-1}^{r-2}\sum_{j=-1}^\ell \psi_j(y)\omega^{j/2}\cdot q_{\ell-j}(x-y)\omega^{(\ell-j)/2}
    \\= B^{\text{reg}}(x-y,s) \Psi^{\text{reg}}(y,s) 
    -\sum_{i=-1}^{r-1} \sum_{j=0}^{r}
    \mathds{1}_{\{i+j \ge  r-1\}}
    \psi_i(y)q_j(x-y)\omega^{(i+j)/2}.
\end{multline*}
As a result we have 
\begin{multline*}
\frac{d}{ds}B^{\text{reg}}(x,s)-
    \Psi^{\text{reg}}(x,s)-
    \sum_{y=1}^{x-1}B^{\text{reg}}(x-y,s) \Psi^{\text{reg}}(y,s) \\ 
    =-\psi_{r-1}(x)(1-s)^{(r-1)/2} + 
    \sum_{y=1}^{x-1}\sum_{i=-1}^{r-1} \sum_{j=0}^{r}
    \mathds{1}_{\{i+j \ge  r-1\}}
    \psi_i(y)q_j(x-y)(1-s)^{(i+j)/2}.
\end{multline*}
It is obvious that the $(n-1)$th coefficient of the generating function on the left hand side equals to 
$$
nb_n(x)-p_n(s)-\sum_{y=1}^{x-1}\sum_{k=1}^{n-1}p^{\text{reg}}_k(y) b^{\text{reg}}_{n-k}(x-y).
$$
Therefore, we need an appropriate estimate for the corresponding coefficient of the function on the right hand side.

If $i\le r-1$, $j\le r$ ans $i+j$ is even, then 
$[s^n](1-s)^{(i+j)/2}=0$ for all $n\ge r>(i+j)/2$.
Consider now the case when  $i+j$ is odd and $i+j\ge r-1$. Then for $x\le \sqrt{n}$ and $0 < y < x$, applying \eqref{eq:theta1}, \eqref{ineq:psi2j} and \eqref{eq:qpower}, we get
\begin{align*}
    |\psi_i(y)q_j(x-y)[s^{n-1}](1-s)^{(i+j)/2}| \le C \frac{(1+x)^{i+j+1}}{n^{(i+j)/2+1}} \le C \frac{(1+x)^{r}}{n^{(r+1)/2}}.
\end{align*}
Summing over $y$ we finally arrive at
\begin{align*}
\left|[s^{n-1}]\sum_{y=1}^{x-1}\sum_{i=-1}^{r-1} \sum_{j=0}^{r}
    \mathds{1}_{\{i+j \ge  r-1\}}
    \psi_i(y)q_j(x-y)(1-s)^{(i+j)/2}\right|
    \le C \frac{(1+x)^{r+1}}{n^{(r+1)/2}}.
\end{align*}
Combining Lemma~\ref{lem:qpower} with the known asymptotics for $a_n^{(j)}$, it is easy to see that
$$
\left|[s^{n-1}\psi_{r-1}(x)(1-s)^{(r-1)2}]\right|
\le C \frac{(1+x)^{r+1}}{n^{(r+1)/2}}.
$$
This completes the proof.
\end{proof}

We next turn to upper bounds for $b_n^{\text{rest}}(x)$.
It turns out that the case $x\ge \sqrt{n}$ is much simpler. 
Indeed, using \eqref{eq:qpower} and the uniform bound $b_n(x) \le \frac{C(1+x)}{n^{3/2}}$ one gets for $n > r$ the following estimate:
\begin{align*}
    |b_n^{\text{rest}}(x)| \le |b_n(x)| + |b_n^{\text{reg}}| &\le
    \frac{C(1+x)}{{n^{3/2}}} + \sum_{j=1}^r |q_j(x)[s^n](1-s)^{j/2}|\\
    &=
    \frac{C(1+x)}{{n^{3/2}}} + \sum_{\ell=1}^{\lf \frac{r+1}{2}\rf} |q_{2\ell-1}(x)a_n^{(\ell+1)}|\\
    &\le
    \frac{C(1+x)}{{n^{3/2}}} + \sum_{\ell=1}^{\lf \frac{r+1}{2}\rf}
    \frac{C(1+x)^{2\ell-1}}{n^{\ell+1/2}} \le \frac{C(1+x)^{r+1}}{n^{(r+3)/2}}.
\end{align*}
For the case $x\le\sqrt{n}$ we have a bit rougher estimate.
\begin{lemma}
\label{lem:b-bound}
For every $n\ge 2$ one has 
    \begin{align*}
b_n^{\text{rest}}(x)\le C(1+x)^{r+1}\frac{\log^{\lf \frac{r}{2}\rf} n}{n^{(r+3)/2}},\quad x\le\sqrt{n}.
    \end{align*}
\end{lemma}
\begin{proof}
It suffices to prove the desired bound for $n$
large enough since for bounded $n$ we have to show only dependence in $x$ and it can be seen the same way as the case $x\ge \sqrt{n}$. Due to this observation, we shall assume that $n>r+2$.

We shall obtain the desired estimate by induction. 
According to Lemma 20 from \cite{VW09},
\begin{equation}
\label{eq:vw-bound}
    b_n(x) \le C \frac{1 + x}{n^{3/2}} . 
\end{equation}
This gives the desired bound for $r=0$.
% Recall that 
% \begin{align*}
% q_0(x)=\sum_{n\ge0}\pr(S_n=x,\tau_0>n)=V(x+1)-V(x)
% \end{align*}
% and 
% $$
% q_1(x)=-\sqrt{\frac{2}{\pi}}V(x),
% $$
% where $V(x)$ is the renewal function corresponding to ascending ladder heights
% $\{\chi^+_k\}$.
% By the local renewal theorem, $V(x+1)-V(x)$ converges, as $x\to\infty$, to 
% $1/\e\chi^+_1$. This implies that
% \begin{align}
% \label{eq:q_0q_1}
% |q_0(x)|\le C\quad\text{and}\quad 
% |q_1(x)|\le C(1+x).
% \end{align}
% Furthermore, we know that $|\psi_j(x)| \le C(1 + x^{j+1})$ for all $j$.
% Plugging these estimates into \eqref{eq:qrecurrent}, one obtains recursively
% \begin{align}
% \label{eq:q>1}
% |q_j(x)| \le C(1+x^{j}) \quad j=2,3,\ldots,r.
% \end{align}
For every $y\in(0,x)$ we split the sum 
$\sum_{k=1}^{n-1}p_k(y)b_{n-k}(x-y)$ into two parts:
\begin{align} \label{eq:seqsplit}
     \sum_{k=1}^{n-1} p_k(y) b_{n-k}(x-y) 
     = \sum_{k=1}^{N-1} p_k(y) b_{n-k}(x-y) 
     + \sum_{k=N}^{n-1} p_k(y) b_{n-k}(x-y),
\end{align}
where $N =  \lceil{n/2}\rceil$.
To estimate the first sum we use the decomposition for $b_n(x-y)$:
\begin{align*}
    \sum_{k=1}^{N-1} p_k(y) b_{n-k}(x-y) 
    = \sum_{k=1}^{N-1} p_k(y) b^{\text{reg}}_{n-k}(x-y) + 
    \sum_{k=1}^{N-1} p_k(y) b^{\text{rest}}_{n-k}(x-y).
\end{align*}
Due to the local central limit theorem, $p_k(y) \le \frac{C}{\sqrt{k}}$ uniformly in $y$. Moreover, it follows from the  induction assumption that 
$$
b_{n-k}^{\text{rest}}(x-y)\le C(x-y)^r\frac{\log^{\lf \frac{r-1}{2}\rf}n}{n^{(r+2)/2}}
$$ 
for every $k < N$. Combining these estimates, we conclude that,
uniformly in $y$,
\begin{align}
\label{eq:p-b^rest}
    \sum_{k=1}^{N-1} p_k(y) b^{\text{rest}}_{n-k}(x-y) 
    \le C
    \frac{(x-y)^{r}\log^{\lf \frac{r-1}{2}\rf}n}{n^{\frac{r+2}{2}}} \sum_{k=1}^{N-1} \frac{1}{\sqrt{k}} 
    \le C
    \frac{x^{r} \log^{\lf \frac{r-1}{2}\rf} n}{n^{\frac{r+1}{2}}}.
\end{align}
Applying the decomposition for $p_k(y)$ and recalling that 
$$
b^{\text{reg}}_{n-k}(x-y)=\sum_{j=0}^r q_j(x-y)[s^{n-k}](1-s)^{j/2},
$$
we get
\begin{align*}
    \sum_{k=1}^{N-1} p_k(y) b^{\text{reg}}_{n-k}(x-y) &=
    \sum_{k=1}^{N-1} p^{\text{reg}}_k(y) b^{\text{reg}}_{n-k}(x-y)
    \\
    &+
    \sum_{j=0}^{r}q_j(x-y)\sum_{k=1}^{N-1} p^{\text{rest}}_k(y) [s^{n-k}](1-s)^{j/2}.
\end{align*}
We first notice that if $n > r$ then $[s^{n-k}](1-s)^{\ell} = 0$ for all integers $0\le\ell\le r/2$ and all $k<N$. Therefore, 
\begin{align*}
&\sum_{j=0}^{r}q_j(x-y)\sum_{k=1}^{N-1} p^{\text{rest}}_k(y) 
[s^{n-k}](1-s)^{j/2}\\
&\hspace{1cm}=\sum_{\ell=1}^{\lf \frac{r+1}{2}\rf}q_{2\ell-1}(x-y)\sum_{k=1}^{N-1} p^{\text{rest}}_k(y) 
[s^{n-k}](1-s)^{\ell-1/2}\\
&\hspace{1cm}=\sum_{\ell=1}^{\lf \frac{r+1}{2}\rf}q_{2\ell-1}(x-y)\sum_{k=1}^{N-1} p^{\text{rest}}_k(y)
a_{n-k}^{(\ell+1)}.
\end{align*}
Set $p_0^{\text{rest}}(y)=0$. 
By the definition, $\sum_{n=1}^\infty p_n^{\text{rest}}(y)s^{n-1}$ belongs to $\cH_{r-1}^0$. Therefore, the function
$\sum_{n=0}^\infty p_n^{\text{rest}}(y) s^n$ also belongs then to the same class $\cH_{r-1}^0$.
Combining \eqref{eq:qpower} and \eqref{eq:p_n^rest}  
and applying Lemma~\ref{lem:partial} with $j =\ell+1$, we get, uniformly in 
$y<x$,
\begin{align*}
    \left|q_{2\ell -1}(x-y))\sum_{k=0}^{N-1} p^{\text{rest}}_k(y) a_{n-k}^{(\ell+1)}\right| 
    \le C
    \frac{x^{2\ell-1}x^{r+1}\log^{\delta_{r-1}}n}{n^{\frac{r + 2\ell + 2}{2}}}.
\end{align*}
Noting that
\begin{align*}
    \frac{x^{r+2\ell}}{n^{\frac{r + 2\ell + 2}{2}}}
    \le \frac{x^{r} }{n^{\frac{r+1}{2}}}
    \quad\text{for}\quad x\le\sqrt{n},
\end{align*}
we conclude that 
\begin{align*}
\left|\sum_{k=1}^{N-1}p_k^{\text{rest}}(y)b_{n-k}^{\text{reg}}(x-y)\right|
\le C \frac{x^{r} \log^{\delta_{r-1}}n}{n^{\frac{r+1}{2}}}
\end{align*}
Combining this with \eqref{eq:p-b^rest}, we arrive at the estimate for $x\le \sqrt{n}$
\begin{align}\label{eq:sumestim1}
\left|\sum_{k=1}^{N-1} p_k(y) b_{n-k}(x-y) 
-\sum_{k=1}^{N-1} p^{\text{reg}}_k(y) b^{\text{reg}}_{n-k}(x-y) \right|
\le C\frac{x^{r}(\log n)^{\lf \frac{r}{2}\rf}}{n^{\frac{r+1}{2}}},
\end{align}
where in the logarithm's power we have used
\begin{align} \label{eq:logpower}
    \left\lf \frac{r}{2} \right\rf =
    \left\lf \frac{r-1}{2}\right\rf + \delta_{r-1}
    \ge 
    \max\left(\left\lf \frac{r-1}{2}\right\rf, \delta_{r-1}\right), \quad \quad r\ge 1.
\end{align}

To estimate the second sum on the right hand side of \eqref{eq:seqsplit}
we first decompose $p_k(y)$:
\begin{align*}
    \sum_{k=N}^{n-1} p_k(y) b_{n-k}(x-y) = 
    \sum_{k=N}^{n-1} p^{\text{reg}}_k(y) b_{n-k}(x-y) 
    +\sum_{k=N}^{n-1} p^{\text{rest}}_k(y) b_{n-k}(x-y).
\end{align*}
Recall that $\sum_{k=1}^\infty b_k(x-y)=q_0(x-y)$. Using \eqref{eq:p_n^rest} and \eqref{eq:qpower}, we obtain
\begin{align*}
\left|\sum_{k=N}^{n-1} p^{\text{rest}}_k(y) b_{n-k}(x-y)\right| 
&\le C\frac{y^{r+1}}{n^{\frac{r+2}{2}}} \sum_{k=1}^{n-N} b_k(x-y) \\
&\le C\frac{y^{r+1}}{n^{\frac{r+2}{2}}} q_0(x-y)
\le C\frac{x^{r+1}}{n^{\frac{r+2}{2}}}
\end{align*}
uniformly in $y<x$.
This implies that
\begin{align}
\label{eq:p^rest-b}
\sum_{k=N}^{n-1} p^{\text{rest}}_k(y) b_{n-k}(x-y)
\le C\frac{x^{r}}{n^{\frac{r+1}{2}}},
\quad x\le\sqrt{n}.
\end{align}

Recall that $b_0(x)=0$. 
Decomposing $b_{n-k}$ and using the definition of $p^{\text{reg}}_k(y)$, 
we have
\begin{align*}
&\sum_{k=N}^{n-1} p^{\text{reg}}_k(y)b_{n-k}(x-y) 
=\sum_{k=N}^{n} p^{\text{reg}}_k(y)b_{n-k}(x-y)\\
&=\sum_{k=N}^{n} p^{\text{reg}}_k(y) b^{\text{reg}}_{n-k}(x-y)
+
\sum_{j=-1}^{r-1}\psi_j(y)\sum_{k=N}^{n} b^{\text{rest}}_{n-k}(x-y) [s^{k}](1-s)^{j/2}\\
&
=\sum_{k=N}^{n} p^{\text{reg}}_k(y) b^{\text{reg}}_{n-k}(x-y)
+\sum_{j=-1}^{r-1}\psi_j(y)\sum_{k=0}^{n-N} b^{\text{rest}}_{k}(x-y) [s^{n-k}](1-s)^{j/2}.
\end{align*}
If $n>r+2$ then $N>(r-1)/2$ and, consequently, $[s^{k}](1-s)^{\ell}=0$ for all $k\ge N$
and all $\ell\le(r-1)/2$. Therefore, 
$$
\sum_{j=-1}^{r-1}\psi_j(y)\sum_{k=0}^{n-N} b^{\text{rest}}_{k}(x-y) [s^{n-k}](1-s)^{j/2}
=\sum_{\ell=0}^{\lf\frac{r}{2}\rf}\psi_{2\ell-1}(y)
\sum_{k=0}^{n-N} b^{\text{rest}}_{k}(x-y)a_{n-k}^{(l+1)}.
$$
Recall that, according to the induction assumption,
$b_{n}^{\text{rest}}(x-y)\le C(x-y)^r\frac{\log^{\lf \frac{r-1}{2}\rf}n}{n^{(r+2)/2}}$ and $\sum_{n=0}^\infty b_n^{\text{rest}}(x)s^n \in \cH_{r-1}^0$.
Then, applying again Lemma \ref{lem:partial} and recalling that 
$|\psi_{2\ell-1}(y)|\le Cy^{2\ell}$, one gets, uniformly in $y<x$, 
\begin{align*}
\left|\psi_{2\ell-1}(y)\sum_{k=0}^{n-N} b^{\text{rest}}_k(x-y)a_{n-k}^{(l+1)}\right|
&\le Cy^{2\ell }(x-y)^r\frac{\log^{\lf \frac{r-1}{2}\rf + \delta_{r-1}}n}{n^{(r+2\ell+1)/2}}
\\ &\le Cx^{2\ell+r }\frac{\log^{\lf \frac{r}{2} \rf}n}{n^{(r+2\ell+1)/2}}.
\end{align*}
Here we have used $\eqref{eq:logpower}$.
Consequently, for $x\le\sqrt{n}$,
$$
\left|\sum_{j=-1}^{r-1}\psi_j(y)
\sum_{k=0}^{n-N} b^{\text{rest}}_k(x-y) [s^{k}](1-s)^{j/2}\right|
\le Cx^{r }\frac{\log^{\lf \frac{r}{2} \rf}n}{n^{(r+1)/2}}.
$$
Combining this with \eqref{eq:p^rest-b}, we conclude that 
\begin{equation}
\label{eq:sumestim2}
\left|\sum_{k=N}^{n-1} p_k(y) b_{n-k}(x-y)-
\sum_{k=N}^{n} p^{\text{reg}}_k(y) b^{\text{reg}}_{n-k}(x-y)\right|
\le Cx^{r }\frac{\log^{\lf \frac{r}{2} \rf}n}{n^{(r+1)/2}},
\end{equation}
for all $x\le\sqrt{n}$. Combining now \eqref{eq:sumestim1} and \eqref{eq:sumestim2},
we obtain, for $x\le\sqrt{n}$,
\begin{align*}
\left|\sum_{k=1}^{n-1}p_k(y) b_{n-k}(x-y)-
\sum_{k=1}^{n} p^{\text{reg}}_k(y) b^{\text{reg}}_{n-k}(x-y)\right|
\le Cx^{r }\frac{\log^{\lf \frac{r}{2} \rf}n}{n^{(r+1)/2}}
\end{align*}
Summing these estimates over $y$ and using relation \eqref{eq:b_n-new} together with the bound \eqref{eq:p_n^rest} (with parameter $r-1$) 
gives 
\begin{align*}
\left|nb_n(x) -p^{\text{reg}}_{n}(x) - 
\sum_{y=1}^{x-1}\sum_{k=1}^{n} p^{\text{reg}}_k(y) b^{\text{reg}}_{n-k}(x-y)\right|
\le Cx^{r+1}\frac{\log^{\lf \frac{r}{2} \rf}n}{n^{(r+1)/2}}.
\end{align*}
So by definition of $b_n^{\text{rest}}$:
\begin{align*}
    |b_n^{\text{rest}}| \le \frac{1}{n}\left|nb_n^{\text{reg}}(x) -p^{\text{reg}}_{n}(x) - 
\sum_{y=1}^{x-1}\sum_{k=1}^{n} p^{\text{reg}}_k(y) b^{\text{reg}}_{n-k}(x-y)\right|
+ Cx^{r+1}\frac{\log^{\lf \frac{r}{2} \rf}n}{n^{(r+3)/2}}.
\end{align*}
Combining the last inequality with the Lemma \ref{lem:bnregrec}
we get the desired estimate for $x\le\sqrt{n}$.
% The case $x \ge \sqrt{n}$ is simpler. Indeed, it is immediate from \eqref{eq:vw-bound}
% that
% \begin{equation*}
% b_n(x)\le\frac{1+x}{n^{3/2}} \le \frac{x^{r+1}}{n^{(r+3)/2}}
% \end{equation*}
% for every $x>\sqrt{n}$.
% Furthermore, using \eqref{eq:qpower}, one gets for every $j\le r$
% the upper bound
% \begin{equation*}
%     \left|q_j(x) [s^n](1-s)^{j/2}\right| 
%     \le C\frac{x^{j}}{n^{(j+2)/2}} 
%     \le C\frac{x^{r+1}}{n^{(r+3)/2}}.
% \end{equation*}
% Consequently,
% \begin{equation*}
% |b^{\text{rest}}_n(x)|=|b_n(x) - \sum_{i=0}^r q_j(x) [s^n](1-s)^{j/2}| \le
% C\frac{x^{r+1}}{n^{(r+3)/2}}.
% \end{equation*}
Thus, the proof is complete.
\end{proof}

\subsection{Asymptotic expansions for $\pr(S_n=x,\overline{\tau}_0>n)$.}
The probabilities $\pr(S_n=x,\overline{\tau}_0>n)$ can be considered in the same way. In this short paragraph we just indicate differences caused by the change of conditioning.

Define 
$$
\overline{b}_n(x):=\pr(S_n=x,\overline{\tau}_0>n),\quad x\ge0.
$$
Instead of \eqref{eq:brecurrent} one has the following recurrent equation
\begin{equation}
\label{eq:brecurrent-a}
\overline{b}_n(x)=\pr(S_n=x)
+\sum_{k=1}^{n-1}\sum_{y=0}^x\pr(S_k=y)\overline{b}_{n-k}(x-y).
\end{equation}
This equality can be proved by the same method, which was used in the proof of Lemma 17 in \cite{VW09}.

Repeating the arguments from Subsection~\ref{subs:expansion}, one gets easily the representation 
\begin{align}
\label{eq:B-bar}
\overline{B}(x,s):=\sum_{n=1}^\infty s^n\overline{b}_n(x)
=\sum_{i=0}^m\overline{q}_i(x)(1-s)^{i/2}+H_{\overline{B}}(x,s),
\end{align}
where the functions $\overline{q}_i(x)$ satisfy the recurrent equations 
\begin{equation}
\label{eq:qrecurrent-a}
-\frac{\ell}{2}\overline{q}_\ell(x)=\psi_{\ell-2}
+\sum_{y=0}^x\sum_{j=-1}^{\ell-2}\psi_j(y)\overline{q}_{\ell-2-j}(x-y),
\quad \ell\ge2
\end{equation}
with starting conditions
$$
\overline{q}_0(x)=\sum_{n=0}^\infty\pr(S_n=x,\overline{\tau}_0>n),
$$
$$
\overline{q}_1(x)=-\sqrt{\frac{2}{\pi}}\sum_{n=0}^\infty\pr(S_n\le x,\overline{\tau}_0>n).
$$
All other arguments from the proofs for $\pr(S_n=x,\tau_0>n)$ can be used without any adaption.
%%%%%%%%%%%%%%%%%%%%%%%%%%%%%%%%%%%%%%%%%%%%%%%%%%%%%%%%%%%%%%%%%%%%%%%%%%%%%%%%%%%%%%%%%%%%%%%%%%%%%%%%%%%%%%%%%%%%%%%%%%%%%%%%%%%%%%%%%%%%%%%%%%%%%
\section{Asymptotic expansions for $\tau_x$.}
\label{sec:tau_x}
\subsection{Left-continuous case.}
For any left-continuous random walk and any $x>0$ one has the equality 
$$
\pr(\tau_x=n)=\frac{x}{n}\pr(S_n=-x),\quad n\ge1.
$$
Since $\E|X_1|^{r+3}$ is assumed finite, we infer from Proposition~\ref{prop:asymplattice} that
$$
\pr(S_n=-x)
=\sum_{j=0}^{\lfloor\frac{r}{2}\rfloor}\theta_j(-x)a_{n-1}^{(j+1)}
+O\left(\frac{1+x^{r+1}}{n^{(r+2)/2}}\right)
$$
uniformly in $x\ge0$. Consequently, 
$$
\pr(\tau_x=n)
=\sum_{j=0}^{\lfloor\frac{r}{2}\rfloor}x\theta_j(-x)\frac{a_{n-1}^{(j+1)}}{n}
+
O\left(\frac{1+x^{r+2}}{n^{(r+4)/2}}\right).
$$
Using the fact that $\frac{a_{n-1}^{(j+1)}}{n}=-a_{n}^{(j+2)}$ for all $n\ge1$, we arrive at the equality
$$
\pr(\tau_x=n)
=\sum_{j=0}^{\lfloor\frac{r}{2}\rfloor}(-x)\theta_j(-x)a_{n}^{(j+2)}
+
O\left(\frac{1+x^{r+2}}{n^{(r+4)/2}}\right).
$$
Summing $\pr(\tau_x=k)$ from $n+1$ to $\infty$, we obtain 
$$
\pr(\tau_x > n)
=\sum_{j=0}^{\lfloor\frac{r}{2}\rfloor}(-x)\theta_j(-x)
\sum_{k=n+1}^\infty a_k^{(j+2)}
+
O\left(\frac{1+x^{r+2}}{n^{(r+2)/2}}\right).
$$
Using now the equality
$$
\sum_{k=n+1}^\infty a_k^{(j+2)}
=\sum_{k=n+1}^\infty \left(a_k^{(j+1)}-a_{k-1}^{(j+1)}\right)=-a_n^{(j+1)},
$$
we finally get 
$$
\pr(\tau_x>   n)
=\sum_{j=0}^{\lfloor\frac{r}{2}\rfloor}x\theta_j(-x)a_n^{(j+1)}
+O\left(\frac{1+x^{r+2}}{n^{(r+2)/2}}\right).
$$
This is the desired equality with $V_j(x)=x\theta_{j-1}(-x)$. It was shown in 
Proposition~\ref{prop:asymplattice} that $\theta_k$ is a polynomial of degree $2k$.
Consequently, $V_j$ has degree $2j-1$. This completes the proof.
\subsection{General case}
Let $m_n$ be the running minimum of the walk $S_n$, that is,
$$
m_n:=\min_{0\le k\le n}S_k.
$$
Define also 
$$
\theta_n:=\max\{k\le n: S_k=m_n\}.
$$
The we have 
$$
\pr(\tau_x>n)=\pr(m_n>-x)
=\sum_{k=0}^n\pr(m_n>-x,\theta_n=k),\quad n\ge 1.
$$
By the duality lemma for random walks,
\begin{align*}
&\pr(m_n>-x,\theta_n=k)\\
&\hspace{1cm}
=\pr(S_j\ge S_k\text{ for all }j\le k,S_k>-x,S_l>S_k\text{ for all }l>k)\\
&\hspace{1cm}=\pr(S_j\ge S_k\text{ for all }j\le k,S_k>-x)\\
&\hspace{3cm}
\pr(X_{k+1}>0,X_{k+1}+X_{k+2}>0,\ldots,X_{k+1}+\ldots+X_n>0)\\
&\hspace{1cm}=
\pr(-S_k<x, -S_k\ge-S_j\text{ for all }j\le k)\pr(\tau_0>n-k)\\
&\hspace{1cm}=\pr(\widetilde{S_k}<x,\widetilde{\tau}_0>k)\pr(\tau_0>n-k),
\end{align*}
where $\widetilde{S}_n=-S_n$, $n\ge0$ and 
$\widetilde{\tau}_0=\inf\{n\ge1:\widetilde{S_n}<0\}$.
Therefore,
$$
\pr(\tau_x>n)
=\sum_{k=0}^n\pr(\widetilde{S_k}<x,\widetilde{\tau}_0>k)\pr(\tau_0>n-k),
\quad n\ge1.
$$
In terms of the generating functions we then have 
\begin{equation}
\label{eq:tau_x-tau_0}
\sum_{n=0}^\infty s^n\pr(\tau_x>n)
=\left(1+\sum_{y=0}^{x-1}\widetilde{B}(s,y)\right)
\sum_{n=0}^\infty s^n\pr(\tau_0>n),
\end{equation}
where 
$$
\widetilde{B}(s,y)
=\sum_{n=1}^\infty s^n\pr(\widetilde{S_n}=y,\widetilde{\tau}_0>n).
$$

Applying \eqref{eq:B-bar} to the random walk $\widetilde{S}_n$, we get
$$
\widetilde{B}(s,y)=\sum_{j=0}^{r}\widetilde{q}_j(y)(1-s)^{j/2}+\widetilde{H}_y(s).
$$
Therefore, 
\begin{align}
\label{eq:sum_B}
\sum_{y=0}^{x-1}\widetilde{B}(s,y)
=\sum_{j=0}^r\widetilde{Q}_j(x)(1-s)^{j/2}+\widetilde{H}(x,s),
\end{align}
where
$$
\widetilde{Q}_j(x):=\sum_{y=0}^{x-1}\widetilde{q}_j(y)
$$
and 
$$
\widetilde{H}(x,s):=\sum_{y=0}^{x-1}\widetilde{H}_y(s).
$$
In Lemma~\ref{lem:b-bound} we have shown that
$$
|[s^n]H_y(s)|\le C(1+y^{r+1})\frac{\log^{\lf\frac{r}{2}\rf} n}{n^{(r+3)/2}}, n\ge 2.
$$
Consequently,
\begin{equation*}
|[s^n]\widetilde{H}(x,s)|\le C(1+x^{r+2})\frac{\log^{\lf\frac{r}{2}\rf} n}{n^{(r+3)/2}}, n\ge 2
\end{equation*}
and
\begin{equation}
\label{eq:H_x_s}
\left|\widetilde{H}(x,s)\right|_{\cH_{r}^{\lf\frac{r}{2}\rf}}\le C(1+x^{r+2}).
\end{equation}
Furthermore, according to \eqref{eq:g.f.tau0},
$$
\sum_{n=0}^\infty s^n\pr(\tau_0>n)=\sum_{k=-1}^{r-1}\mu_{k+1}(1-s)^{k/2}
+H_0(s),\quad H_0(s) \in \cH^1_{r-1}.
$$
Plugging this equality and \eqref{eq:sum_B} into \eqref{eq:tau_x-tau_0}, we obtain
\begin{equation}
\label{eq:tau_x-repr}
\sum_{n=0}^\infty s^n\pr(\tau_x>n)
=\sum_{j=-1}^{r-1} Q_j(x)(1-s)^{j/2}+H(x,s),
\end{equation}
where
\begin{align}
\label{eq:Q-q}
Q_j(x)=\mu_{j+1}+\sum_{k=-1}^j\mu_{k+1}\widetilde{Q}_{j-k}(x),\quad
j=1,2,\ldots,r-1
\end{align}
and 
\begin{align*}
H(x,s)&
=H_0(s)\left(1+\sum_{j=0}^{r}\widetilde{Q}_j(x)(1-s)^{j/2}+\widetilde{H}(x,s)\right)\\
&\hspace{1cm}+\widetilde{H}(x,s)\sum_{k=-1}^{r-1}\mu_{k+1}(1-s)^{k/2}\\
&\hspace{1cm}+\sum_{j=r}^{2r-1}\sum_{k=j-r}^j\mu_{k+1}\widetilde{Q}_{j-k}(x)(1-s)^{j/2}.
\end{align*}
It follows from Lemma~\ref{1-slemma} that 
$$
\widetilde{H}_1(x,s):=\sum_{j=r}^{2r-1}\sum_{k=j-r}^j\mu_{k+1}\widetilde{Q}_{j-k}(x)(1-s)^{j/2}
\in\cH_{r-1}^0.
$$
Since $\widetilde{q}_j(x)\le C(1+x^j)$, $\widetilde{Q}_j(x)\le C(1+x^{j+1})$. This implies that
\begin{align}
\label{eq:H1-norm}
\left|\widetilde{H}_1(x,s)\right|_{\cH_{r-1}^0}
\le C(1+x^{r+1}).
\end{align}
Applying Lemmata~\ref{mlemma} and \ref{lem:hlemma_2}, we infer that 
$$
\widetilde{H}_2(x,s):=
H_0(s)\left(1+\sum_{j=0}^{r}\widetilde{Q}_j(x)(1-s)^{j/2}+\widetilde{H}(x,s)\right)
\in\cH_{r-1}^0
$$
and, due to \eqref{eq:H_x_s},
\begin{align}
\label{eq:H2-norm}
\left|\widetilde{H}_2(x,s)\right|_{\cH_{r-1}^0}
\le C(1+x^{r+2}).
\end{align}
Using Lemma \ref{lem:hlemma_2} once again, we conclude that
$$
\widetilde{H}_3(x,s):=
\widetilde{H}(x,s)\sum_{k=0}^{r-1}\mu_{k+1}(1-s)^{k/2}
\in\cH_{r}^{\lf\frac{r}{2}\rf}
$$
and
\begin{align}
\label{eq:H3-norm}
\left|\widetilde{H}_3(x,s)\right|_{\cH_{r}^{\lf\frac{r}{2}\rf}}
\le C(1+x^{r+2}).
\end{align}
Finally, according to Lemma~\ref{lem:hlemma_3} and to \eqref{eq:H_x_s},
$$
\widetilde{H}_4(x,s):=\mu_0\widetilde{H}(x,s)(1-s)^{-1/2}\in\cH_{r-1}^{\lf\frac{r}{2}\rf+\delta_r}
$$
and 
\begin{align}
\label{eq:H4-norm}
\left|\widetilde{H}_4(x,s)\right|_{\cH_{r-1}^{\lf\frac{r}{2}\rf+\delta_r}}
\le C(1+x^{r+2})
\end{align}
Combining \eqref{eq:H1-norm} --- \eqref{eq:H4-norm}, we conclude that $H(x,s)$
belongs to $\cH_{r-1}^{\lf\frac{r}{2}\rf+\delta_r}$ and that the corresponding norm is bounded by $C(1+x^{r+2})$. Thus, applying now Lemma~\ref{lem:rdecomp} to the representation \eqref{eq:tau_x-repr} and noting that $\lf \frac{r+1}{2}\rf = \lf \frac{r}{2} \rf + \delta_r$, we finish the proof.
\section{Polyharmonicity of coefficients}
\label{sec:poly}
\subsection{Proof of the Theorem~\ref{thm:polyharm}}
Applying the total probability law, we obtain
\begin{align} \label{eq:taumarcovpr}
    \pr (\tau_x > n+1) = \sum_{y=1}^{\infty} \pr(X_1 = y-x) \pr(\tau_y > n).
\end{align}
The assumption $\E |X_1|^r  < \infty $ implies that 
\begin{align*}
    \pr(X_1 > \sqrt{n}) = o(n^{-r/2}).
\end{align*}
Then, for every fixed $x$,
\begin{align} \label{eq:taumarcovtail}
\nonumber
\sum_{y \ge \sqrt{n}}\pr(X_1 = y-x) \pr(\tau_y > n) 
&\le \sum_{y = \lfloor \sqrt{n} \rfloor + 1}^{\infty}
\pr(X_1 = y-x) \\
&= \pr(X_1 + x \ge\sqrt{n}) = o(n^{-r/2}).
\end{align}
We next apply the asymptotic expansion from Theorem \ref{thm:tau_x} 
to $\pr(\tau_x > n+1)$  and  to $\pr(\tau_y > n)$ with $y\le\sqrt{n}$.
That expansion can be written in the following way:
\begin{equation*}
\pr(\tau_y > n) 
= \sum_{j=1}^{\lf\frac{r}{2}\rf+1}V_j(y) a_{n}^{(j)} + R_{y,n},
\end{equation*}
where $R_{y,n}$ is the remainder term.
Plugging this and \eqref{eq:taumarcovtail} into \eqref{eq:taumarcovpr},
we obtain
\begin{align*}
&\sum_{j=1}^{\lf\frac{r}{2}\rf+1}V_j(x) a_{n+1}^{(j)} + R_{x,n+1}\\
&\hspace{1cm}= 
\sum_{y=1}^{\lf \sqrt{n} \rf} \pr(X_1 = y-x)
\left( \sum_{j=1}^{\lf\frac{r}{2}\rf+1}V_j(y) a_{n}^{(j)} + R_{y,n}
\right) + o(n^{-r/2}).
\end{align*}
Changing the order of summation leads then to the equality
\begin{align*}
    \sum_{j=1}^{\lf\frac{r}{2}\rf+1}V_j(x) a_{n+1}^{(j)} = 
\sum_{j=1}^{\lf\frac{r}{2}\rf+1}a_{n}^{(j)}\sum_{y=1}^{\lf \sqrt{n} \rf} V_j(y)\pr(X_1 = y-x)
+ R_n + o(n^{-r/2}).
\end{align*}
where
\begin{align*}
    R_n = -R_{x,n+1} + \sum_{y=1}^{\lf \sqrt{n} \rf}\pr(X_1 = y-x) R_{y,n}.
\end{align*}
According to Theorem \ref{thm:tau_x}, 
$$
R_{y,n}=O\left(\log^{r/2-2}n\frac{1+|y|^{r-2}}{n^{(r-2)/2}}\right)
$$
Then there exists a constant $C$ such that
\begin{align*}
    \frac{1}{C}|R_n| &\le \log^{r/2-2}n\frac{(1+|x|)^{r-2}}{n^{(r-2)/2}} + \sum_{y=1}^{\lf \sqrt{n} \rf} \pr(X_1 = y-x) \log^{r/2-2}n\frac{(1+|y|)^{r-2}}{n^{(r-2)/2}} \\
    &\le \frac{\log^{r/2-2}n}{n^{(r-2)/2}} 
    \left(
    (1+|x|)^{r-2} 
    +
    \sum_{y=1}^{\lf \sqrt{n} \rf} \pr(X_1 = y-x) (1+|y|)^{r-2}
    \right).
\end{align*}
Since $\E X_1^{r-2}$ is finite, we conclude that
\begin{equation*}
    |R_n| = O\left(\frac{\log^{r/2-2}n}{n^{(r-2)/2}} \right).
\end{equation*}
And hence we obtain
\begin{align} \label{eq:taumarkov2}
\sum_{j=1}^{\lf\frac{r}{2}\rf+1}V_j(x) a_{n+1}^{(j)} = 
\sum_{j=1}^{\lf\frac{r}{2}\rf+1}a_{n}^{(j)}\sum_{y=1}^{\lf \sqrt{n} \rf} V_j(y)\pr(X_1 = y-x)
+ O\left(\frac{\log^{r/2-2}n}{n^{(r-2)/2}} \right).
\end{align}

We next notice that
$$
\sum_{y=1}^{\lf \sqrt{n} \rf} V_j(y)\pr(X_1 = y-x)
=\E[V_j(x+X);x+X>0]-\E[V_j(x+X);x+X>\sqrt{n}].
$$
Recalling that $V_j(x)\le C_j (1+x)^{2j-1}$ and using the Markov inequality, one gets 
$$
\E[V_j(x+X);x+X>\sqrt{n}]=o(n^{j-1/2-r/2})
$$
for every $j\le (r+1)/2$.
Recalling that $a_n^{(j)} \sim n^{-j+1/2}$, we finally obtain
\begin{align*}
    a_n^{(j)}\E[V_j(x+X);x+X>\sqrt{n}]= o\left({n^{-r/2}}\right).
\end{align*}
Applying this to \eqref{eq:taumarkov2} we have
\begin{align*}
\sum_{j=1}^{\lf\frac{r}{2}\rf+1}V_j(x) a_{n+1}^{(j)} = 
\sum_{j=1}^{\lf\frac{r}{2}\rf+1}a_{n}^{(j)}
\E [ V_j(x+X_1), x+X_1 >0 ]
+ O\left(\frac{\log^{r/2-2}n}{n^{(r-2)/2}} \right).
\end{align*}
Let us remind relation (\textit{i}) form Lemma \ref{lem:an} for the elements $a_n^{j}$:
\begin{align*}
    a_n^{(j)} = a_{n+1}^{(j)} - a_{n+1}^{(j+1)}
\end{align*}
Then one has
\begin{align}
\label{eq:polyharm.rel}
\nonumber
\sum_{j=1}^{\lf\frac{r}{2}\rf+1}V_j(x) a_{n+1}^{(j)} = 
\sum_{j=1}^{\lf\frac{r}{2}\rf+1} &(a_{n+1}^{(j)} - a_{n+1}^{(j+1)})
\E [V_j(x+X_1), x+X_1 > 0]\\
&+ O\left(\frac{\log^{r/2-2}n}{n^{(r-2)/2}} \right).
\end{align}
Dividing both sides by $a_{n+1}^{(1)}$ and letting $n\to\infty$, we conclude that
$V_1(x)=\E[V_1(x+X);x+X>0]$. In other words, $V_1$ is a harmonic function. Using this
equality we can rewrite \eqref{eq:polyharm.rel}:
\begin{align*}
\sum_{j=2}^{\lf\frac{r}{2}\rf+1}V_j(x) a_{n+1}^{(j)} = 
-a_{n+1}^{(2)}V_1(x)+
\sum_{j=2}^{\lf\frac{r}{2}\rf+1} &(a_{n+1}^{(j)} - a_{n+1}^{(j+1)})
\E \big( V_j(x+X_1), x+X_1 \ge 0 \big)\\
&+ O\left(\frac{\log^{r/2-2}n}{n^{(r-2)/2}} \right).
\end{align*}
Dividing this time both sides by $a_{n+1}^{(2)}$ and letting again $n\to\infty$, we get 
$$
V_2(x)=-V_1(x)+\E[V_2(x+X);x+X>0].
$$
Using the definition of the operator $P$, we can rewrite this equality in the following way:
$$
PV_2(x)-V_2(x)=V_1(x).
$$
Recalling that $V_1$ is harmonic, we also have the equality 
$$
(P-I)^2V_2=0.
$$
Thus, $V_2$ is polyharmonic of order 2. Repeating  this argument $k$ times we infer that $V_k$ is polyharmonic of order $k$.

The functions $U_k$ can be considered in the same manner. 
Instead of \eqref{eq:taumarcovpr} one can use the following obvious equality
\begin{align*}
\pr(S_{n+1}=x,\tau_0>n+1)
&=\sum_{y=1}^\infty\pr(S_n=y,\tau_0>n)\pr(y+X_{n+1}=x)\\
&=\sum_{y=1}^\infty\pr(x-X=y)\pr(S_n=y,\tau_0>n).
\end{align*}
Applying now the asymptotic expansions from Theorem~\ref{thm:cond_prob}
and repeating the arguments from the proof of polyharmonicity of functions $V_k$,
one gets easily the desired property of functions $U_k$.

\subsection{Proof of Theorem~\ref{thm:V-asymp}}
We start by considering functions $\psi_j(y)$. Recall that they are defined by the 
decomposition
\begin{align}
\label{eq:psi-repr}
\sum_{n=1}^\infty s^{n-1}\pr(S_n=y)=\sum_{j=-1}^{{m}}\psi_j(x)(1-s)^{j/2}+H_\Psi(x,s), \quad H_\Psi\in\cH_{{m}}^0.
\end{align}
\begin{lemma}\label{lem:psipoly}
For every $j\le m/2$ there exists a polynomial $P^\psi_j(x)$ of degree {$2j+1$} such that
\begin{align*}
    \psi_{2j}(x) = P_j^\psi(x) + O(e^{-\varepsilon x}).
\end{align*}
\end{lemma}
\begin{proof}
Let $\Delta$ denote the standard difference operator, that is, 
\begin{align*}
    \Delta G (x) =  G (x+1) - G(x)
\end{align*}
for every function $G$. It is clear that if $P(x)$ is a polynomial of degree $m$ then 
$\Delta^{m+1}P(x)\equiv 0$.
Thus, the lemma will be proven if we show that
\begin{align}
\label{eq:delta-bound}
    \Delta^{2j+2}\psi_{2j}(x) = O(e^{-\varepsilon x}). 
\end{align}
Differentiating \eqref{eq:psi-repr} $j$ times, we obtain
\begin{align*}
    \psi_{2j}(x) = \frac{(-1)^j}{j!}\sum_{n=j+1}^\infty 
    \frac{(n-1)!}{(n-1-j)!}\big(
    \pr( S_n = x) - \sum_{i=0}^{j}\psi_{2i-1}(x)a_{n-1}^{(j+1)}
    \big).
\end{align*}
We already know from Proposition \ref{prop:asymplattice} that $\psi_{2\ell-1}(x) = \theta_\ell(x)$ is a polynomial of degree $2\ell$. Therefore,
\begin{align*}
   \Delta^{2j+2}\psi_{2j}(x) 
   =\frac{(-1)^j}{j!}
   \sum_{n=j+1}^\infty 
   \frac{(n-1)!}{(n-1-j)!}
    \Delta^{2j+2}\pr( S_n = x).
\end{align*}
Let $f(t)$ be characteristic function of $X_1$. Then, due to the inversion formula,
\begin{align*}
    \pr(S_n = x) = \frac{1}{2\pi} \int_{-\pi}^{\pi} {e^{-itx}}f^n(t)dt.
\end{align*}
This implies that
\begin{align*}
    \frac{(n-1)!}{(n-1-j)!}\Delta^{2j+2}\pr(S_n = x) 
    &= \frac{1}{2\pi} \int_{-\pi}^{\pi} \Delta^{2j+2}{e^{-itx}} 
    \frac{(n-1)!}{(n-1-j)!}f^n(t)dt.
\end{align*}
Consequently,
\begin{align*}
    \Delta^{2j+2}\psi_{2j}(x) 
    &=\frac{(-1)^j}{j!} \sum_{n=j+1}^\infty \frac{1}{2\pi} \int_{-\pi}^{\pi} \Delta^{2j+2}{e^{-itx}} \frac{(n-1)!}{(n-1-j)!}f^n(t)dt\\
    &= \frac{(-1)^j}{j!}\frac{1}{2\pi} \int_{-\pi}^{\pi} \Delta^{2j+2}{e^{-itx}} \left( \sum_{n=j}^\infty \frac{n!}{(n-j)!}f^{n+1}(t) \right) dt.
\end{align*}
% The desired is equivalent to showing that for all $\ell \le j$
% \begin{align} \label{eq:deltaintegral1}
%     \frac{1}{j!}\int_{-\pi}^{\pi} \Delta^{2j+2}{e^{-itx}} \left( \sum_{n=\ell}^\infty \frac{n!}{(n-\ell)!}f^{n-\ell}(t) \right) dt = O(e^{-\varepsilon x})
% \end{align}
% since all terms $\Delta^{2j+2}\psi_{2j}(x)$ just a linear combinations of expressions from the LHS. 
We next notice that 
\begin{align*}
\sum_{n=j}^\infty \frac{n!}{(n-j)!} \lambda^{n-j} 
= \frac{d^j}{d\lambda^j}\left( \sum_{n=0}^\infty \lambda^n \right)^{(j)}= \left(\frac{1}{1-\lambda}\right) 
= \frac{j!}{(1-\lambda)^{j+1}}.
\end{align*}
Using this equality, we finally get the following representation
\begin{align} \label{eq:deltaphi}
    % \frac{1}{j!}\int_{-\pi}^{\pi} \Delta^{2j+2}{e^{-itx}} &\left( \sum_{n=j}^\infty \frac{n!}{(n-j)!}f^{n-j}(t) \right) dt
    % \\&
    \nonumber
    \Delta^{2j+2}\psi_{2j}(x)
    &=\frac{(-1)^j}{2\pi}\int_{-\pi}^{\pi}
    \Delta^{2j+2}e^{-itx}\frac{f^{j+1}(t)}{(1-f(t))^{j+1}} dt\\
    &= \frac{j!}{2\pi}\int_{-\pi}^{\pi}e^{-itx} \frac{g_{2j+2}(t)f^{j+1}(t)}{(1-f(t))^{j+1}} dt
\end{align}
where
\begin{align*}
    g_m(t) := e^{itx} \cdot \Delta^{m}{e^{-itx}} =   \sum_{y=0}^{m} (-1)^y\binom{m}{y}e^{-ity}.
\end{align*}
We next show that the function in the integral on the right hand side of \eqref{eq:deltaphi} is analytic in some neighborhood $\Omega$ of the set
\begin{align*}
    \Pi := \big\{t \big| \Re t \in [-\pi, \pi], \; \Im t \in [-\varepsilon, 0] \big\}
\end{align*}
for every sufficiently small $ \varepsilon > 0$. 
First, We notice that $f(t)$ is analytic in $\Omega$ for sufficiently small $\varepsilon$ due to the assumption $\E e^{\lambda|X|}<\infty$ for some $\lambda > 0$.
Second, we have
\begin{align*} %\label{eq:analytic}
    e^{-itx}\frac{g_{2j+2}(t)}{(1-f(t))^{j+1}} =
    e^{-itx}
    \cdot
    \frac{g_{2j+2}(t)}{t^{2j+2}} \cdot \left(\frac{t^{2}}{1-f(t)}\right)^{j + 1}
\end{align*}
Since $f(t)$ is a characteristic function of the discrete random variable with zero mean we have that $(1-f(t))/t^2$ is analytic and does not have any zeroes in $\Omega$ for sufficiently small $\varepsilon$. Therefore its inverse is also analytic in $\Omega$. 
As a result we have the desired analiticity. Consequently,
% Next near zero we have
% \begin{align*}
% f(t) = 1 - \E X^2 t^2 + o(t^2), \quad t\to 0.
% \end{align*}
% Hence the function $\frac{1-f(t)}{t^2}$ is also analytic in the same region. Also
% \begin{align*}
%     g_m(t) = c_m t^{m} + o(t^{m}), \quad t\to 0.
% \end{align*}
% Combining those properties we obtain in case $\ell = j$
% \begin{align*}
%     \frac{g_{2j+2}(t)}{(1-f(t))^{j+1}} = \frac{c_{2j+2}}{(\E X^2)^{j+1}} + o(1), \quad t\to 0.
% \end{align*}
% For the case $\ell < j$ it holds
% \begin{align*}
%     \frac{g_{2j+2}(t)}{(1-f(t))^{\ell+1}} = O(t^{2(j-\ell)}), \quad t\to 0.
% \end{align*}
% We know also that 
% \begin{align*}
%     f(t) \neq 1, \quad t \in [-\pi, 0)\cup (0,\pi].
% \end{align*}
% So the function $e^{-itx}g_{2j+2}(t)/{(1-f(t))^{j+1}}$ is continuous and analytic on the set $\Pi$ for some sufficiently small $\varepsilon$. Hence
\begin{align*}
    \int_{\partial \Pi} e^{-itx}\frac{g_{2j+2}(t)f^{j+1}(t)}{(1-f(t))^{j+1}} dt = 0
\end{align*}
We can rewrite the last integral into 4 integrals over the sets:
\begin{align*}
    \Gamma_1 &= \big\{t\big|\Re t \in [-\pi,\pi], \; \Im t =0\big\},\\
    \Gamma_2 &= \big\{t\big|\Re t = \pi, \; \Im t \in [-\varepsilon,0]\big\},\\
    \Gamma_3 &= \big\{t\big|\Re t \in [-\pi,\pi], ; \Im t = -\varepsilon \big\},\\
    \Gamma_4 &= \big\{t\big|\Re t = -\pi, \; \Im t \in [-\varepsilon, 0]\big\}.
\end{align*}
Since the function in the integral is $2\pi$-periodic, the integrals over $\Gamma_2$ and $\Gamma_4$ cancel each other out. This implies that
\begin{align*}
\left| \int_{-\pi}^{\pi}e^{-itx}
\frac{g_{2j+2}(t)f^{j+1}(t)}{(1-f(t))^{j+1}} dt \right| 
    &= \left|\int_{-\pi}^\pi
    e^{-i(t-\varepsilon i)x}
    \frac{g_{2j+2}(t-\varepsilon i)f^{j+1}(t-\varepsilon i)}{(1-f(t-i\varepsilon))^{j+1}} dt\right|\\
    &=e^{-\varepsilon x}\int_{-\pi}^\pi
     e^{-itx}
     \left|
    \frac{g_{2j+2}(t-\varepsilon i)f^{j+1}(t-\varepsilon i)}{(1-f(t-i\varepsilon))^{j+1}}\right|dt\\
    &\le e^{-\varepsilon x} \int_{-\pi}^\pi
    \left|
    \frac{g_{2j+2}(t-\varepsilon i)f^{j+1}(t-\varepsilon i)}{(1-f(t-i\varepsilon))^{j+1}}
    \right|
    dt \le e^{-\varepsilon x}.
\end{align*}
%It means that
%\begin{align*}
    %\left|\Delta^{2j+2}\psi_{2j}(x)\right| \le C(j) e^{-\varepsilon x} = O(e^{-\varepsilon x})
%\end{align*}
Plugging this into \eqref{eq:deltaphi}, we conclude that \eqref{eq:delta-bound} holds. Thus, the proof of the lemma is complete.
\end{proof}

In the lemma we analyse the properties of polynomial approximation under the convolution.
\begin{lemma}\label{lem:explop}
Let $P$ and $Q$ be polynomials. Assume that functions $u, v$ satisfy the representations
\begin{align} \label{eq:lemexppol}
    u(x) = P(x) + O(e^{-\varepsilon x}), \quad v(x) = Q(x) + O(e^{-\varepsilon x}),
\end{align}
where $\varepsilon>0$.
Then there exist a polynomial $R$ of degree  $\deg P + \deg Q + 1$ and a natural number $r\ge1$ such that
\begin{align*}
    \sum_{y=1}^{x-1}u(y)v(x-y) = R(x) + O(x^re^{-\varepsilon x}).
\end{align*}
\end{lemma}
\begin{proof} 
Define
$$
h_u(x)=u(x)-P(x)\quad\text{and}\quad h_v(x)=v(x)-Q(x).
$$
The assumption \eqref{eq:lemexppol} implies that, for all $x>0$ and some constant $C$,
\begin{equation}
\label{eq:exp-appr}
|h_u(x)|\le Ce^{-\varepsilon x}
\quad\text{and}\quad
|h_v(x)|\le Ce^{-\varepsilon x}.
\end{equation}
Then we have 
\begin{align*}
\sum_{y=1}^{x-1}u(y)v(x-y)
=&\sum_{y=1}^{x-1}P(y)Q(x-y)+\sum_{y=1}^{x-1}h_u(y)Q(x-y)\\
&+\sum_{y=1}^{x-1}P(y)h_v(x-y)+\sum_{y=1}^{x-1}h_u(y)h_v(x-y).
\end{align*}
Let $m$ and $n$ denote the degrees of polynomials $P$ and $Q$ respectively.
Then, clearly, $\sum_{y=1}^{x-1}P(y)Q(x-y)$ is a polynomial of degree $n+m+1$. Furthermore, \eqref{eq:exp-appr} implies that 
\begin{align*}
\left|\sum_{y=1}^{x-1}h_u(y)h_v(x-y)\right| 
\le C^2(x-1)e^{-\varepsilon x}.
\end{align*}
Therefore, it remains to show that the sums $\sum_{y=1}^{x-1}h_u(y)Q(x-y)$, $\sum_{y=1}^{x-1}P(y)h_v(x-y)$ can also be approximated by appropriate polynomials. Due to the symmetry, it suffices to consider one sum only. Moreover, it suffices to consider monomials. For every $k\ge 1$ we have 
\begin{align*}
\sum_{y=1}^{x-1}(x-y)^kh_u(y)
&=\sum_{j=0}^k{k \choose j}x^{k-j}(-1)^j\sum_{y=1}^{x-1}y^jh_u(y)\\
&=\sum_{j=0}^k{k \choose j}x^{k-j}(-1)^j\sum_{y=1}^{\infty}y^jh_u(y)
-\sum_{j=0}^k{k \choose j}x^{k-j}(-1)^j\sum_{y=x}^{\infty}y^jh_u(y).
\end{align*}
The first sum over $j$ is a polynomial of degree $k$. To estimate the second sum, we notice that \eqref{eq:exp-appr} yields
$$
\left|\sum_{y=x}^{\infty}y^jh_u(y)\right|
\le \widehat{C} x^je^{-\varepsilon x}.
$$
Therefore, 
$$
\sum_{j=0}^k{k \choose j}x^{k-j}(-1)^j\sum_{y=x}^{\infty}y^jh_u(y)
=O\left(x^ke^{-\varepsilon x}\right).
$$
This completes the proof of the lemma and show that one can take $r=\max\{n,m,1\}$.
\end{proof}

Using this lemma we can now consider polynomial approximations for the functions $q_\ell$ and $\overline{q}_\ell$ defined in \eqref{eq:qrecurrent} and \eqref{eq:qrecurrent-a} respectively.
\begin{lemma} \label{lem:qalmpoly}
For every $\ell\ge1$ there exist polynomials $Q_\ell(x)$, $\overline{Q}_\ell(x)$ of degree $\ell$ such that
\begin{align*}
    &q_{\ell}(x) = Q_{\ell}(x) + O(x^\ell e^{-\varepsilon x}),\\
    &\overline{q}_{\ell}(x) = \overline{Q}_{\ell}(x) + O(x^\ell e^{-\varepsilon x})
\end{align*}
for some $\varepsilon>0$.
\end{lemma}
\begin{proof}
We consider the functions $q_\ell$ only.

Recall that 
$$
q_0(x)=\sum_{n=1}^\infty\pr(S_n=x,\tau_0>n)=V(x+1)-V(x)
$$
and 
$$
q_1(x)=-\sqrt{\frac{2}{\pi}}V(x),
$$
where $V(x)$ is the renewal function of ladder heights. It is known that the assumption 
$\E e^{\lambda|X|}<\infty$ implies that $V(x)=ax+b+O(e^{-\varepsilon x})$ for appropriate
constants $a$ and $b$. This implies the desired approximations for $q_0(x)$ and $q_1(x)$.

Assume now that $\ell\ge2$. Since $\psi_{-1}(x) \equiv \frac{1}{\sqrt{2\pi}}$ we  can rewrite \eqref{eq:qrecurrent} as follows
\begin{equation*}
-\frac{\ell}{2} q_\ell(x)=\psi_{\ell-2}(x) 
+
\frac{1}{\sqrt{2\pi}}\sum_{y=1}^{x-1} q_{\ell-1}(x-y) + 
\sum_{y=1}^{x-1}\sum_{j=0}^{\ell-2}\psi_j(y)q_{\ell-2-j}(x-y).
\end{equation*}
This equality allows one to use the induction. Assume that the lemma holds for $q_{j}$ with $j\le\ell-1$. Then, using Lemma~\ref{lem:explop} with $u(y)\equiv1$ and 
$v(y)=Q_{\ell-1}(y)$, we conclude that there exists a polynomial $P_\ell$ of degree $\ell$ such that
\begin{align*}
\sum_{y=1}^{x-1} q_{\ell-1}(x-y)=P_\ell(x)
+O\left(x^{\ell}e^{-\varepsilon x}\right).
\end{align*}

If the index $j$ is odd then $\psi_j(x)$ is a polynomial of degree $j+1$.
Moreover, if $j$ is even, then one can approximate $\psi_j(x)$ by a polynomial of degree $j+1$. This allows us to apply again Lemma~\ref{lem:explop} and to conclude that
\begin{align*}
    \sum_{y=1}^{x-1}\psi_j(y)q_{\ell-2-j}(x-y) = R(x) + O(x^\ell e^{-\varepsilon x})
\end{align*}
for some polynomial $R$ of degree $(j + 1) + (\ell -2 - j) + 1 = \ell$. 
This complets the proof.
\end{proof}
Recalling that
\begin{align*}
U_j(x) = q_{2j-1}(x)\quad\text{and}\quad \overline{U}_j(x) = \overline{q}_{2j-1}(x)
\end{align*}
and applying Lemma~\ref{lem:qalmpoly}, we get the desired approximations for the functions
$U_j$ and $\overline{U}_j$. To obtain the polynomial approximations for $V_j$, it is sufficient to recall that $V_j=Q_{2j-3}$ with $Q_\ell$ defined in \eqref{eq:Q-q} and to apply once again Lemma~\ref{lem:qalmpoly}. 

\section{Proofs of Propositions~\ref{prop:S_n<0} and~\ref{prop:asymplattice}}
\label{prop8.and.9}

\begin{proof}[Proof of Proposition~\ref{prop:S_n<0}]
Assume first that the distribution of $X_1$ is non-lattice and that 
$$
\limsup_{|t|\to\infty}\left|\e e^{itX_1}\right|<1.
$$
In view of existence of $(r+3)$th   moment
it follows  from Theorem VI.3.4. in Petrov's book \cite{Petrov} that 
\begin{equation}
\label{eq:asymp1}
\pr(S_n\le x\sigma\sqrt{n})=\Phi(x) 
+\sum_{\nu=1}^{r + 1} \frac{Q_{\nu}(x)}{n^{\nu/2}}
+o\left( \frac{1}{n^{(r+1)/2}} \right).
\end{equation}
uniformly in $x\in\R$. Here $\Phi(x)$ denotes the distribution function of the standard normal distribution. The coefficients $Q_\nu(x)$ can be expressed in terms of cumulants and Hermite polynomials, see~\cite[Chapter 6, (1.13)]{Petrov},
\begin{equation}\label{qformula}
Q_\nu(x) = -\frac{1}{\sqrt{2\pi}} e^{-x^2/2} \sum H_{\nu + 2s -1}(x) \prod_{m=1}^\nu \frac{1}{k_m!}
\left(\frac{\gamma_{m+2}}{(m+2)!\sigma^{m+2}}\right)^{k_m},
\end{equation}
where the sum is taken over all integer non-negative solutions $(k_1, k_2, \dots, k_\nu)$ to the equation $k_1 + 2k_2 + \dots \nu k_\nu = \nu$
and $s$ denotes the sum of all $k_i$'s, that is, $s = k_1 + k_2 + \dots + k_\nu$.
Here $\gamma_\nu$ is the cumulant of order $\nu$ of the random variable $X_1$, and $H_m(x)$ is the $m$th Hermite polynomial:
\begin{equation*}
H_m(x) = (-1)^m e^{x^2/2}\frac{d^m}{d x^m}e^{-x^2/2}.
\end{equation*}
For these polynomials one has also the following representation: 
\begin{equation} \label{hformula}
H_m(x) = m! \sum_{k=0}^{\lfloor \frac{m}{2} \rfloor} 
\frac{(-1)^k x^{m-2k}}{k!(m-2k)!2^k}.
\end{equation}
Since $e^{-x^2/2}$ is even, Hermite polynomials $H_m(x)$ are even for even $m$ and odd for odd $m$. 
Then it follows from the representation~\eqref{qformula} that 
$Q_\nu(x)$ is even for odd $\nu$ and odd for even $\nu$. 
Consequently, $Q_\nu(0)=0$ for all even indices $\nu$.
From this observation and from \eqref{eq:asymp1} we infer that
\begin{align*}
\pr(S_n\le 0)
&=\frac{1}{2}+\sum_{\nu={1}}^{r+1}\frac{Q_\nu(0)}{n^{\nu/2}}
+
o\left( \frac{1}{n^{(r+1)/2}} \right)\\
&=\frac{1}{2}+\sum_{j=1}^{\lfloor r/2\rfloor+1}
\frac{Q_{2j-1}(0)}{n^{j-1/2}}
+o\left( \frac{1}{n^{(r+1)/2}} \right)
\end{align*}
and, consequently,
\begin{align}
\label{eq:asymp2} 
\frac{\Delta_n}{n}=-\sum_{j=1}^{\lfloor r/2\rfloor +1}
\frac{Q_{2j-1}(0)}{n^{j+1/2}}+o\left( \frac{1}{n^{(r+3)/2}}\right).
\end{align}

We next derive a similar decomposition in the lattice case. Here, 
in view of existence of $(r+3)$th   moment 
we can apply Theorem VI.3.6. in \cite{Petrov} to get the following expression 
\begin{multline}
\label{eq:asymp1discrete}
\pr(S_n\le x\sigma\sqrt{n})=U_r(x)
+ \sum_{\nu=1}^{r + 1}
\delta_\nu \left( \frac{1}{\sigma \sqrt{n}}\right)^\nu 
S_\nu \left(x\sigma \sqrt{n} \right)\frac{d^\nu}{dx^\nu}U_r(x)
\\+ o\left( \frac{1}{n^{(r+1)/2}} \right),
\end{multline}
where
\begin{equation*}
    U_r(x) = \Phi (x) + \sum_{j=1}^{r + 1} \frac{Q_{j}(x)}{n^{j/2}},
\end{equation*}

\begin{equation*}
    S_{2k}(x) = 2\sum_{\ell=1}^\infty \frac{\cos 2\pi \ell x}{(2\pi\ell)^{2k}}, \quad S_{2k+1}(x ) 
    = 2\sum_{\ell=1}^\infty \frac{\sin 2\pi\ell x}{(2\pi\ell)^{2k+1}},
\end{equation*}
and
\begin{equation*}
    \delta_\nu = 
    \begin{cases}
    +1, \text{ if }\nu = 4m+1, 4m+2,\\
    -1, \text{ if }\nu = 4m+3, 4m.
    \end{cases}
\end{equation*}
In case of $x=0$ we then have
\begin{equation*}
    S_{2k}(0) = 2\sum_{\ell=1}^\infty \frac{1}{(2\pi\ell)^{2k}}, \quad 
    S_{2k+1}(0) = 0.
\end{equation*}
Our next goal is to show that for $\nu = 0, \dots,  r + 1$ the expression
\begin{equation} \label{expression}    
\left( \frac{1}{\sigma \sqrt{n}}\right)^\nu 
\left[\frac{d^\nu}{dx^\nu}U_r(x)\right]_{x=0}
\end{equation}
contains only odd powers of $\sqrt{n}$. 
For that note that 
when $\nu$ is even 
and $j$ is odd 
\[
\left[\frac{d^\nu}{dx^\nu}Q_j(x)\right]_{x=0} = 0 
\]
since $Q_j(x)$ is an even function. 
Similarly, 
when $\nu$ is odd 
and $j$ is even the same equality holds.

%The case $\nu=0$ has already been verified in the non-lattice case. 
%We next notice that $U_r(x)$ contains only summands of the form
%\(
    %e^{-x^2/2}\frac{x^\ell }{n^{m/2}}
%\)
%where $m$ and $\ell$ have different parity. Next 
%note that $e^{-x^2/2} x^\ell$ is an even function for even $l$ and 
%odd for odd $l$. Hence
%\begin{equation*}
%\frac{1}{\sqrt{n}} \left[\frac{d^\nu}{dx^\nu}e^{-x^2/2}\frac{x^\ell}{n^{m/2}}\right]_{x=0} = 0
     %\frac{d}{dx} \left[e^{-x^2/2}\frac{x^\ell}{n^{m/2}} \right] = \ell e^{-x^2/2}\frac{x^{\ell-1}}{n^{(m+1)/2}} -e^{-x^2/2}\frac{2x^{\ell+1}}{n^{(m+1)/2}}.
%\end{equation*}
%for all odd $\nu$ and even $l$, and for all even $\nu$ and odd $l$. 
%One can see that under this operation in every terms powers of $x$ and $\sqrt{n}$ keep having different parity. 
Hence after substituting $x=0$ we obtain only odd powers of $\sqrt{n}$ and, consequently, 
\begin{align}
\label{eq:asymp 3}
\frac{\Delta_n}{n} = \frac{1}{2} - \pr(S_n\le 0) 
&=\sum_{j=1}^{\lfloor r/2\rfloor+1}
\frac{\widetilde{Q}_{{j}}}{n^{j+1/2}}
+o\left( \frac{1}{n^{(r+3)/2}} \right),
\end{align}
where $\widetilde{Q}_{j}$ are some real numbers. 
Using the asymptic expansion~\eqref{eq:anasympdecomp} 
with  \eqref{eq:asymp2} and  \eqref{eq:asymp 3}, we obtain the desired representation for $\Delta_n/n$.
\end{proof}

\begin{proof}[Proof of Proposition~\ref{prop:asymplattice}]
In the lattice case, by Theorem VII.3.13. in \cite{Petrov} with $k = r+3$ we have,
uniformly in $x$,
\begin{align} \label{eq:asymp4}
    \sigma \sqrt{n} p_n(x) = \frac{1}{\sqrt{2\pi}} e^{-t^2/2} + \sum_{\nu=1}^{r+1} \frac{q_{\nu}(t)}{n^{\nu/2}} + o\left(\frac{1}{n^{(r+1)/2}}\right),
\end{align}
where
$$
t=\frac{x}{\sigma \sqrt{n}}
$$
and 
\begin{align*}
    q_{\nu}(t) = \frac{1}{\sqrt{2\pi}} e^{-t^2/2}\sum H_{\nu + 2s}(t) \prod_{m=1}^\nu \frac{1}{k_m!}
\left(\frac{\gamma_{m+2}}{(m+2)!\sigma^{m+2}}\right)^{k_m}.
\end{align*}
As before, the sum is taken over all integer non-negative solutions $(k_1, k_2, \dots, k_\nu)$ to the equation $k_1 + 2k_2 + \dots \nu k_\nu = \nu$
and $s$ denotes the sum of all $k_i$'s, that is, $s = k_1 + k_2 + \dots + k_\nu$. Furthermore, $\gamma_\nu$ is the cumulant of order $\nu$ of the random variable $X_1$ and $H_m(t)$ is the $m$-th Hermite polynomial.

Before we start to analyse the behaviour of functions $q_\nu (t)$, we show that \eqref{eq:asymp4} holds also in the absolute continuous case. By Theorem VII.3.15. in \cite{Petrov} again with $k=r+3$ we have
\begin{align*}
    \frac{d}{dt} \pr \bigg(\frac{1}{\sigma \sqrt{n}}\sum_{i=1}^n X_i \in (0,t) \bigg)
     = \frac{1}{\sqrt{2\pi}} e^{-t^2/2} + \sum_{\nu=1}^{r+1} \frac{q_{\nu}(t)}{n^{\nu/2}} + o\left(\frac{1}{n^{(r+1)/2}}\right).
\end{align*}
For $t=x/ (\sigma \sqrt{n})$ we have
\begin{equation*}
    \frac{d}{dt} \pr \big(\frac{1}{\sigma \sqrt{n}}\sum_{i=1}^n X_i \in (0,t) \big)
    = \sigma \sqrt{n} \frac{d}{dx}\bigg( \pr \big(S_n \in (0,x) \big)\bigg)'_x.
\end{equation*}
This gives one validity of \eqref{eq:asymp4} in the absolute continuous case. 

Due to the asymptotic expansion of the binomial coefficients~\eqref{eq:anasympdecomp} in order to prove the proposition it suffices to show that for all
$\nu, s\ge0$ there exist polynomials $\xi_{j,\nu,s}$ of degree at most $2j$ such that
\begin{align} \label{eq:onetermdec}
    \left|
    \exp\left(-\frac{x^2}{2\sigma n}\right) H_{\nu + 2s}
    \left(\frac{x}{\sigma \sqrt{n}}\right) n^{-{\nu/2}} - \sum_{j=0}^{\lf 
r/2\rf}\frac{\xi_{j,\nu,s}(x)}{n^{{j}}} 
\right|
\le
C_{\nu,s,r}\frac{(1+|x|)^{r+1}}{n^{(r+1)/2}}.
 \end{align}
(Note that $H_0(x) = 1$ so the term $e^{-t^2/2}/\sqrt{2\pi}$  is also covered.) Indeed, we can straightforwardly sum up such estimates to get the desired expansion for $\sigma \sqrt{n}p_n(x)$. Then dividing it by $\sigma \sqrt{n}$ and changing the basis from $n^{-(j+1/2)}$ to $a_{n-1}^{(j+1)}$ one obtains \eqref{eq:prop6}.

It follows from \eqref{hformula} that
\begin{equation*}
H_m\left(\frac{x}{\sigma \sqrt{n}}\right) 
=m! \sum_{k=0}^{\lf \frac{m}{2} \rf }\frac{1}{n^{m/2-k}}
    \frac{(-1)^k x^{m-2k}}{k!(m-2k)!2^k\sigma^{m-2k}}
    =:
    \sum_{k=0}^{\lf \frac{m}{2} \rf }
    \frac{w_{k,m}(x)}{n^{m/2-k}}.
\end{equation*}
Therefore,
\begin{equation} \label{eq:hermitpol1}
    \frac{H_{\nu + 2s}(\frac{x}{\sigma \sqrt{n}})}{n^{\nu/2}} 
    =
    \frac{1}{n^{\nu/2}}
    \sum_{k=0}^{\lf \frac{\nu + 2s}{2} \rf }
    \frac{w_{k,\nu + 2s}(x)}{n^{\nu/2 +s-k}}
    =
    \sum_{k=0}^{\lf \frac{\nu + 2s}{2} \rf }
    \frac{w_{k,\nu + 2s}(x)}{n^{\nu + s-k}}.
\end{equation}
Here notice that in the last formula we have monomials $x^a/n^b$ with $b=(a+\nu)/2$. This implies that the left hand side is bounded by some constant for all $|x|\le \sqrt{n}$.

% So, we have only integer powers of $n$ in the last formula.
% As a result we have
% \begin{align} \label{eq:qnupoly}
% \frac{q_{\nu}(\frac{x}{\sigma \sqrt{n}})}{n^{\nu/2}}
% =
% e^{-{\frac{x^2}{2\sigma^2n}}}
% \sum_{k=1}^{2\nu} \frac{q_{\nu,k}(x)}{n^{k}}
% \end{align}
% where $q_{\nu, k}(x)$ is a polynomial of degree not higher than $2k$.
Fix some $m\ge 1$. By the series representation for $e^{-t}$, for every $t\ge 0$ there exists $\theta=\theta(t) \in [0,t]$ such that
\begin{align*}
    e^{-t} = \sum_{k=0}^m \frac{(-1)^k t^k}{k!} + \frac{(-1)^{m+1}e^{-\theta}}{(m+1)!}t^{m+1}.
\end{align*}
Therefore,
\begin{align*}
    \left|e^{-t} - \sum_{k=0}^m \frac{(-1)^k t^k}{k!}
    \right|
    \le \frac{t^{m+1}}{(m+1)!}.
\end{align*}
Let
\begin{align*}
    y_m(t) = e^{-t} - \sum_{k=0}^m \frac{(-1)^k t^k}{k!}.
\end{align*}
Consequently,
\begin{align} \label{ineq:restexptaylor}
    \left|y_m\left(
    -{\frac{x^2}{2\sigma^2n}}
    \right)\right|
    \le \frac{1}{2^{m+1}\sigma^{2m+2}(m+1)!}\frac{x^{2m+2}}{n^{m+1}}.
\end{align}
Setting $d_k = (-1)^k/\big(k!(2\sigma^2)^k\big)$, we arrive at the following representation:
\begin{align} \label{eq:exptailor}
    \exp \left(
    -\frac{x^2}{2\sigma^2 n}
    \right)
    = \sum_{k=0}^m d_k\frac{x^{2k}}{ n^k} + y_m\left(
    -{\frac{x^2}{2\sigma^2n}}
    \right).
\end{align}
Substituting \eqref{eq:exptailor} and \eqref{eq:hermitpol1} into \eqref{eq:onetermdec}, we obtain
\begin{align*}
    \exp\left(-\frac{x^2}{2\sigma n}\right) H_{\nu + 2s}&
    \left(\frac{x}{\sigma\sqrt{n}}\right) n^{-\nu} \\&=
    \left(
    \sum_{k=0}^m d_k\frac{x^{2k}}{ n^k} + y_m\left(
    -{\frac{x^2}{2\sigma^2n}}
    \right)
    \right)
    \left(
    \sum_{k=0}^{\lf \frac{\nu + 2s}{2} \rf }
    \frac{w_{k,\nu + 2s}(x)}{n^{\nu + s-k}}
    \right).
\end{align*}
As we have already mentioned, the second factor s bounded by a constant in the case when $|x| \le \sqrt{n}$. Consequently, there exists a constant $C$ such that
\begin{align*}
    \left|
    y_m\left(
    -{\frac{x^2}{2\sigma^2n}}
    \right)
    \left(
    \sum_{k=0}^{\lf \frac{\nu + 2s}{2} \rf }
    \frac{w_{k,\nu + 2s}(x)}{n^{\nu + s-k}}
    \right)\right|
    \le
    C \frac{x^{2m+2}}{n^{m+1}},
    \quad |x|\le\sqrt{n} .
\end{align*}
By the same reason there are some polynomials $\xi_{j,\nu,s}$ of degree $2j$ such that, uniformly in $|x| \le \sqrt{n}$,
\begin{align*}
    \left|
    \left(
    \sum_{k=0}^m d_k\frac{x^{2k}}{ n^k}
    \right)
    \left(
    \sum_{k=0}^{\lf \frac{\nu + 2s}{2} \rf }
    \frac{w_{k,\nu + 2s}(x)}{n^{\nu + s-k}}
    \right)
    -
    \sum_{j=0}^{m}\frac{\xi_{j,\nu,s}(x)}{n^{j}}
    \right| 
    \le
    C
    \frac{(1+|x|)^{2m+2}}{n^{m+1}}.
\end{align*}
In order to obtain \eqref{eq:onetermdec} we choose
\begin{align*}
    m = \lf r/2 \rf :=
    \begin{cases}
        (r-1)/2, \quad &r \equiv 1\mod2,\\
        r/2,  &\text{else.}
    \end{cases}
\end{align*}
For odd $r$ we immediately got \eqref{eq:onetermdec} for $|x|\le \sqrt{n}$ and for even $r$ we additionally apply estimation
\begin{align*}
    \frac{(1+|x|)^{r+2}}{n^{(r+2)/2}}
    \le \frac{1+|x|}{\sqrt{n}}\frac{(1+|x|)^{r+1}}{n^{(r+1)/2}} \le 2 \frac{(1+|x|)^{r+1}}{n^{(r+1)/2}}, \quad n\ge1.
\end{align*}
So now we can conclude that one has
\begin{align*}
p_n(x)
=
\sum_{j=0}^{\lf 
r/2\rf}\theta_j(x) a_{n-1}^{(j+1)}
+
h_n(x),
\end{align*}
where $\theta_j$ is a polynomial of degree $2j$ as a linear combination over $\nu$ and $s$ of $\xi_{j,\nu,s}$ and for $|x| \le \sqrt{n}$ it holds
\begin{align} \label{eq:restestim}
    |h_n(x)|\le C_r \frac{(|x|+1)^{r+1}}{n^{(r+2)/2}}, \quad n\ge 1.
\end{align}

Now, the final step is to show validity of \eqref{eq:restestim}
in case $|x| \ge \sqrt{n}$. One can write
\begin{align*}
    \left| p_n(x)  - \sum_{j=0}^{\lf 
    r/2\rf}
    \theta_j(x) a_{n-1}^{(j+1)}\right|
    \le |p_n(x)| + \sum_{j=0}^{\lf 
    r/2\rf}
    \left|\theta_j(x) a_{n-1}^{(j+1)}\right|.
\end{align*}
For the first term one infers {from the local central limit theorem} that
\begin{align*}
    |p_n(x)| \le \frac{c}{n^{1/2}} \le \frac{c x^{r+1}}{n^{(r+2)/2}}.
\end{align*}
Moreover, for every $j$ and all $x$ with $|x|\ge\sqrt{n}$ one has the estimate
\begin{align*}
    \left|\theta_j(x) a_{n}^{(j+1)}\right| \le 
    \left|\theta_j(x) \right| \cdot \left|a_{n-1}^{(j+1)}\right|\le c_j \frac{|x|^{2j}}{n^{j+1/2}} \le c_j\frac{|x|^{r+1}}{n^{(r+2)/2}}.
\end{align*}
This gives the desired bound in case $|x|\ge \sqrt{n}$. The proof is then complete.
\end{proof}
 
\end{document}